\documentclass[11pt]{article}
\usepackage{amsmath,amssymb,amsthm,graphicx,subfigure,float,url}
\usepackage[colorlinks=true,citecolor={Plum},linkcolor={Periwinkle}]{hyperref}
\usepackage{pdfsync}
\usepackage{float}
\usepackage[left=2cm,right=2cm,top=1.5cm,bottom=1.5cm]{geometry}
\usepackage{mathrsfs}
\usepackage[usenames, dvipsnames]{xcolor}
\usepackage{dsfont}
\title{\sf Is the Faber-Krahn inequality true for the Stokes operator?}
\author{Antoine Henrot\footnote{Universit\'e de Lorraine, CNRS, Institut Elie Cartan de Lorraine, BP 70239 54506 Vand\oe uvre-l\`es-Nancy Cedex, France ({\tt antoine.henrot@univ-lorraine.fr}).}
\and Idriss Mazari-Fouquer\footnote{CEREMADE, UMR CNRS 7534, Universit\'e Paris-Dauphine, Universit\'e PSL, Place du Mar\'echal De Lattre De Tassigny, 75775 Paris cedex 16, France, (\texttt{mazari@ceremade.dauphine.fr}) }
\and Yannick Privat\footnote{Universit\'e de Lorraine, CNRS, Institut Elie Cartan de Lorraine, Inria, BP 70239 54506
Vandœuvre-l\`es-Nancy Cedex, France. ({\tt yannick.privat@univ-lorraine.fr}).}~\footnote{Institut Universitaire de France (IUF)}.
}

\def\u {{\mathbf{u}}}
\def\V{{\mathbf{V}}}
\def\U{{\mathbf{U}}}
\def\n{{\nabla}}
\def\B{{\mathbb{B}}}
\def\O{{\Omega}}
\def\div{{\mathrm{div}}}
\def\curl{{\mathbf{curl}}}
 \def\u{{\mathbf{u}}}
  \def\v{{\mathbf{v}}}

 \def\er{{\vec{e}_r}}
  \def\et{{\vec{e}_\theta}}
 \def\ep{{\vec{e}_\phi}}

\newcommand{\N} {\mathbb{N}}
\newcommand{\Z} {\mathbb{Z}}
\newcommand{\C} {\mathbb{C}}
\newcommand{\R} {\mathbb{R}}

\newcommand{\e} {\varepsilon}

\renewcommand{\geq}{\geqslant}
\renewcommand{\leq}{\leqslant}

\newtheorem{theorem}{Theorem}  
\newtheorem{proposition}{Proposition}
\newtheorem{corollary}{Corollary}
\newtheorem{definition}{Definition}
\newtheorem{lemma}{Lemma}

\theoremstyle{definition}\newtheorem{remark}{Remark}

\begin{document}

\maketitle

\begin{abstract}
The goal of this paper is to investigate the minimisation of the first eigenvalue of the (vectorial) incompressible Dirichlet-Stokes operator. After providing an existence result, we investigate optimality conditions and we prove the following surprising result: while the ball satisfies first and second-order optimality conditions in dimension 2, it does not in dimension 3, so that the Faber-Krahn inequality for the Stokes operator is probably true in $\R^2$, but does not hold in $\R^3$. The multiplicity of the first eigenvalue of the Dirichlet-Stokes operator in the ball in $\R^3$
plays a crucial role in the proof of that claim.
\end{abstract}

\paragraph{Keywords:} Shape derivation, Shape semi-differentiation, Spectral optimisation, Stokes operator.

\paragraph{AMS Classification:} 47A10, 49Q10, 76D07.

\section{Introduction}

{In this article, we focus on a spectral optimisation problem governed by the Stokes operator; the latter is crucial in the analysis of fluid motions. In order to motivate this question, let us observe that the eigenvalues of the Stokes operator can be physically interpreted as the decay frequencies of the eigenmodes of a fluid. Each associated eigenmode represents a specific fluid motion. The eigenvalues are also related to the characteristic time scales of the dynamics of the fluid. Furthermore, it should be noted that the eigenvalues of the Stokes operator also appear naturally when dealing with the long-time behaviour of solutions to the (nonlinear) Navier-Stokes equations \cite{Berchio_2024,zbMATH03901316}. Our main goal here will thus be to further the understanding of the influence of the geometry of the domain on the first Stokes eigenvalue.
}

\subsection{Setting}

\paragraph{Scope of the paper.}
The optimisation of spectral quantities with respect to a domain is a central question in shape optimisation and in the calculus of variations. 
{Throughout the article, $\Omega$ will denote a subset of the ambient space $\R^d$, with $d\in \{2,3\}$. The specific choice of dimension $d$ will be specified for each result.}
Of particular importance is the paradigmatic question of minimising (with respect to the domain $\O$) the first eigenvalue of a differential operator under a volume constraint on $\O$. The celebrated Faber-Krahn inequality asserts that, when the operator is the (scalar) Dirichlet-Laplacian, the ball minimises the first eigenvalue with a volume constraint. 
In the present paper, we investigate the minimisation of the first eigenvalue of the Dirichlet-Stokes operator. Namely, {let $\O$ be a bounded, smooth open set, and let $W^{1,2}_0(\O;\R^d)$ be defined as 
the usual Sobolev space (here, the functions are $\R^d$-valued; see also the Notations paragraph below). Consider} the first eigenvalue
\begin{equation}\label{Eq:Rayleigh} 
\lambda_1(\O):=\min_{\substack{\u\in W^{1,2}_0(\O;\R^d)\\ \nabla\cdot \u=0\text{ in }\O, \ \u \neq 0}}\frac{\int_\O \Vert \n \u\Vert^2}{\int_\O \Vert \u\Vert^2},\end{equation} where $\nabla\cdot$ stands for the divergence operator. 
{In the expression above,  $\n\u$ stands for the Jacobian matrix of $\u$ or, in other words,
\[\n \u=\left(\frac{\partial u_i}{\partial x_j}\right)_{1\leq i,j\leq d}.\] 
With a slight abuse of notation, we use $\Vert \cdot\Vert$ to denote not only the Euclidean norm of $\R^d$, as in the term $\Vert \u\Vert^2$, but also the Frobenius norm of $\nabla \u$,, that is, 
\[
\Vert \nabla \u\Vert ^2=\n\u \colon\n\u=\sum_{i,j=1}^d\left(\frac{\partial u_i}{\partial x_j}\right)^2.
\]
}
This eigenvalue is associated with the eigen-equation
\begin{equation}\label{Eq:Eigenequation}
\begin{cases}-\Delta \u+\n p=\lambda_1(\O)\u &\text{ in }\O\,, 
\\ \n \cdot\u=0&\text{ in }\O\,, 
\\ \u=0&\text{ on }\partial\O.\end{cases}\end{equation}
In \eqref{Eq:Eigenequation} the function {$p$ (which is unique up to an additive constant) is the pressure associated with $\u$} , which can be interpreted as the Lagrange multiplier attached to the incompressibility constraint { $\n\cdot\u=0$}. 

The problem under consideration in this article is the following:
\begin{equation}\label{Eq:PvIntro}\fbox{$\displaystyle
\inf_{\O\subset \R^d\,, \O\text{ bounded, }|\O|\leq V_0}\lambda_1(\O).$}\end{equation} The constant $V_0>0$ is a given volume constraint (which will be immaterial as we will work with a scale invariant formulation of \eqref{Eq:PvIntro}).
  
The main contributions of the article are the following:
 \begin{enumerate}
 \item First, we obtain an \textbf{existence result} in the class of quasi-open sets (this is the natural framework for such  existence results; we refer to Definition \ref{De:QuasiOpen}) This is Theorem \ref{Th:Existence}. The method of proof relies on the concentration-compactness principle of Lions \cite{LionsCC}, which was adapted by Bucur \cite{zbMATH01454740} to the setting of shape optimisation. The main difficulty here is to handle the incompressibility condition.
  \item Second, we investigate the \textbf{local optimality} of the ball by checking first and second-order optimality conditions for Hadamard variations. In Theorem \ref{Th:Optimality2d}, we prove that the ball satisfies these optimality conditions in $\R^2$. On the other hand, we prove in Theorem \ref{Th:NonOptimality3d} that the ball does not satisfy first-order optimality conditions in $\R^3$. The proof of the non-optimality of the ball in $\R^3$ relies on the fact that $\lambda_1$ is, in this case, a multiple eigenvalue.
  \item We then derive precise \textbf{necessary optimality conditions} for the minimality of a set in $\R^3$; this is Theorem \ref{Th:Necessary}; this theorem and its proof are linked to some recent results in the optimisation of the first $\curl$ eigenvalue \cite{Gerner_2023}.
 \end{enumerate}

We discuss related works in details in section \ref{Se:Biblio} but let us already highlight some aspects of our problem. One of the first results in spectral shape optimisation is the celebrated Faber-Krahn inequality \cite{faber1923beweis,Krahn_1925}, which asserts that, under a volume constraint, the ball minimises the first eigenvalue of the Dirichlet-Laplacian in any dimension. This inequality can be derived through numerous methods, some of which can be generalised to the minimisation of other spectral quantities; the ball is very often the minimiser (or at least a local minimiser) for the first eigenvalue of several scalar differential operators.  Thus, Theorem \ref{Th:NonOptimality3d} hints at a deeper discrepancy between the scalar and the vectorial case. In general, let us observe that the literature devoted to spectral optimisation problems in the vectorial case is scarce. To the best of our knowledge, the results closest to ours were derived very recently by Enciso, Gerner \& Peralta-Salas \cite{enciso2022optimal,zbMATH07697301,Gerner_2023} in the case of the $\curl$ operator, a problem which was also investigated by Cantarella, DeTurck, Gluck \& Teytel  in the early 2000's \cite{Cantarella_2000_Bis}.

\paragraph{Notations.}
Throughout the paper we use the following conventions:
\begin{enumerate}
\item For any $k\in \N\,, p\in [1;+\infty)$, $W^{k,p}(\O)$ (resp. $W^{k,p}_0(\O)$) denotes the usual Sobolev space of order $k$ and index $p$ (resp. the functions of $W^{k,p}(\O)$ whose trace on $\partial \O$ is zero). Likewise, $W^{k,p}(\O;\R^d)$ (resp. $W^{k,p}_0(\O;\R^d)$) denotes the  Sobolev space of order $k$ and index $p$ of vector-valued functions, each coordinate of which is a $W^{k,p}(\O)$ (resp. $W^{k,p}_0(\O)$) function.
\item $\mathscr C^\infty_c(\R^d)$ is the set of compactly supported $\mathscr C^\infty$ functions.
\item $\n$ is the gradient operator, $\n \cdot$ the divergence operator and $\curl $ the curl operator.
\item All bold letters will be used to designate a $\R^2$ or $\R^3$ vector or vector field.
\item The double-dot product of two matrices $A=(a_{ij})_{1\leq i,j\leq N}$ and $B=(b_{ij})_{1\leq i,j\leq N}$ is the real number $A\colon B$ given by $A\colon B=\sum_{i,j=1}^Na_{ij}b_{ij}$.
\item $\mathbb B_d=\{ x\in\R^d\ \vert\ \Vert x\Vert< 1\}$ is the Euclidean unit ball of $\R^d$.
\item {$\Vert \mathbf{X}\Vert$ denotes either the Euclidean norm of $\R^d$ if $\mathbf{X}$ is a vector, or the Frobenius norm of $\mathbf{X}$ if $\mathbf{X}$ is a square matrix of size $d$; in other words $\Vert \mathbf{X}\Vert^2=\mathbf{X}:\mathbf{X}$. } 
\item $\mathscr{C}^\infty(\Omega)$ denotes the space of infinitely differentiable real-valued functions in $\Omega$.  
\item $\operatorname{Per}(\Omega)$  denotes the De Giorgi perimeter of $\Omega$.\end{enumerate}

\paragraph{Precise statement of the problem.}
To set our problem in an appropriate (from the point of view of existence properties) framework, we recall the definition of quasi-open sets \cite[Chapter 3]{zbMATH06838450}:
\begin{definition}\label{De:QuasiOpen}
A subset $\O$ of $\R^d$ is called \textbf{quasi-open} if there exists a non-increasing sequence of open sets $\{\omega_k\}_{k\in \N}$ such that 
\[ \lim_{k\to  \infty}\mathrm{cap}(\omega_k)=0\quad \text{ and }\quad \forall k\in \N,\quad  \Omega\cup \omega_k\text{ is open},\] where $\mathrm{cap}$ denotes the capacity of an open set. Recall that it is defined as 
\[ \mathrm{cap}(\omega):=\sup_{\substack{K\textnormal{compact}\\ K\subset \omega}}\, \inf_{\substack{v\in \mathscr C^\infty_c(\R^d)\\ v\geq 1\text{ on }K}}\int_{\R^d} \left\vert \n v\right\vert^2.\]

We define 
\[ \mathscr O:=\{\O\subset \R^d\,, \O\textnormal{ quasi-open}\,, |\O|>0\}.\] 
\end{definition}
Quasi-open sets are the natural framework in which to consider spectral shape optimisation problems \cite{Buttazzo_1993,Velichkov_2015}. For any given  $\O\in \mathscr O$ we can define \cite[Chapter 3]{zbMATH06838450} the Sobolev space $W^{1,2}_0(\O)$ as the set of functions $v\in W^{1,2}(\R^d)$ that are equal to 0 on $\R^d\backslash \O$ up to a set of zero capacity. 
For any $\O\in \mathscr O$ we set 
\begin{equation}\label{Eq:RayleighStokes}
\lambda_1(\O):=\min_{\substack{\u\in W^{1,2}_0(\O;\R^3)\color{black}\\ \nabla\cdot \u=0\text{ in }\O\,, \u \neq 0}}\frac{\int_\O \Vert \n \u\Vert^2}{\int_\O \Vert \u\Vert^2}.\end{equation}
It is clear that for any $t>0$ and any $\O\in \mathscr O$ we have $\lambda_1(t\O)=\frac1{t^2}\lambda_1(\O)$, and we thus introduce the following scale invariant functional:
\[ \mathcal F:\mathscr O\ni \O\mapsto |\O|^{\frac2d}\lambda_1(\O).\] The shape optimisation problem under consideration is 
\begin{equation}\label{Eq:Pv}\fbox{$\displaystyle
\inf_{\O\in \mathscr O}\left(\mathcal F(\O):=|\O|^{\frac2d}\lambda_1(\O)\right).$}\end{equation}

\subsection{Main results of the paper}

\paragraph{Existence of an optimal shape.}
We begin with the following:
\begin{theorem}\label{Th:Existence}
The variational problem \eqref{Eq:Pv} has a solution $\O^*$.
\end{theorem} 
The proof of this theorem hinges, in this vectorial setting, on the approach to the Faber-Krahn inequality of Bucur \& Freitas \cite[Proposition 3.1]{BucurFreitas} or Bucur \& Varchon \cite{BucurVarchon}, which in turn relies on the concentration-compactness principle of Lions \cite{LionsCC}.  However, because of the incompressibility constraint, we have to modify the concentration-compactness approach, which in turns prohibits us from guaranteeing that $\O^*$ is bounded.\color{black}

Of course, if $\mathscr O$ were to be replaced with $\mathscr O_D:=\{\O\in \mathscr O\,, \mathscr O\subset D\}$ for a given compact set $D$ (``box constraint"), the existence of a minimiser would follow from an adaptation of the Buttazzo-Dal Maso theorem \cite{Buttazzo_1993}, although one should be cautious when handling the zero-divergence constraint. For the sake of completeness, we include this result in Appendix \ref{Ap:DalMaso}. We refer to section \ref{Se:Biblio} for more references on shape optimisation without box constraints.

\begin{remark}
When working in the two-dimensional case, if the domain $\O$ is simply connected,  we can introduce a potential $\psi$ such that the velocity
$\u$ can be written $\u=(-\partial_y \psi, \partial_x \psi)^\top $, so that the first eigenvalue of the Stokes operator coincides with the first eigenvalue of the buckling problem \cite[Chapter 11]{Henrot_2006}.  {The function  $\psi$  is called the {\it stream function};  we refer for instance to \cite[Lemma~2.5]{zbMATH03906935} for its existence and uniqueness. Its existence is a special case of the Helmoltz-Hodge decomposition in dimension 2.} Our existence result is thus linked to the theorem of Ashbaugh \& Bucur \cite{Ashbaugh_2003}, which asserts the existence of an optimal domain for the buckling problem in the class of simply connected domains. We will comment more on this aspect of the problem in section \ref{Se:Biblio}.
\end{remark}

\paragraph{Optimality conditions and (semi-)differentials of eigenvalues.}

The local optimality of a shape will be investigated by means of Hadamard perturbations \cite[Chapter 5]{zbMATH06838450}. This already requires some regularity of the shape under consideration. To be more precise: if $\O$ is a bounded domain with $\mathscr C^2$ boundary, if $\Phi\in W^{3,\infty}(\R^d;\R^d)$ we define, for any $t\in (-1;1)$ small enough, 
\[ \O_{t\Phi}:=(\mathrm{Id}+t\Phi)\O.\] This definition is meaningful as, for any $t$ small enough, $(\mathrm{Id}+t\Phi)$ is a smooth diffeomorphism. For a given shape functional $F$ and a given shape $\O$, we say that $F$ is differentiable at $\O$ if, for any $\Phi\in  W^{3,\infty}(\R^d;\R^d)$, the limit
\[ \langle dF(\O),\Phi\rangle:=\lim_{t\to 0}\frac{F(\O_{t\Phi})-F(\O)}t\] exists, and if it is a linear form in $\Phi$. In this case, this limit is called the first-order shape derivative of $F$ at $\O$ in the direction $\Phi$.

 The functional $F$ might not be differentiable (typically, when considering the shape derivative of a multiple eigenvalue, as can be the case here), but we can usually define the \textbf{semi-differential} of $F$ at $\O$ in the direction $\Phi$ as
 \[
 \langle \partial F(\O),\Phi\rangle:=\lim_{t\searrow 0}\frac{F(\O_{t\Phi})-F(\O)}t.
 \] 
{ whenever the limit above exists, and defines a linear form in $\Phi\in W^{3,\infty}(\R^d;\R^d)$.}
 
 Such semi-differentials are of particular importance when dealing with multiple eigenvalues. A systematic approach was developed by several authors to handle such situations; while we refer to section \ref{Se:Biblio} for precise references, let us point to \cite{Caubet_2021}, which is the most related to our work as it deals with semi-differentials for the linear elasticity system.
 
The first-order optimality conditions for \eqref{Eq:Pv} read: if $\O^*$ is a smooth optimal set  then, for any smooth vector field $\Phi$, there holds
\begin{equation}\label{Eq:FirstOrderOC}
\langle dF(\O^*\color{black}),\Phi\rangle\geq\color{black} 0\text{ if $F$ is differentiable, }\quad \langle \partial F(\O^*),\Phi\rangle\geq 0\text{ if $F$ is semi-differentiable.}\end{equation}
{The first condition above, called the {\it Euler inequality}, can often be reduced to an equality depending on the nature of the admissible set of shapes. This will be made precise in what follows.}

 We say that $F$ is twice differentiable at $\O$ if it is differentiable and if, for any $\Phi\in W^{3,\infty}(\R^d;\R^d)$ the limit
\[
\langle d^2F(\Omega){\Phi},{\Phi}\rangle=\lim_{t\to 0}\frac{ F(\O_{t\Phi})-F(\O)-t\langle dF(\Omega),{\Phi}\rangle}{t^2}
\] exists.

\paragraph{Local optimality of $\B_2$ in $\R^2$.} The main result here is the following:
\begin{theorem}\label{Th:Optimality2d}
The first eigenvalue $\lambda_1(\B_2)$ is simple. Consequently, $\mathcal F$ is twice differentiable at $\B_2$ and there holds:
\begin{enumerate}
\item For any $\Phi \in W^{3,\infty}(\R^d,\R^d)$, 
\[ \langle d\mathcal F(\B_2),\Phi\rangle=0.\]
\item There exists a constant $c_0>0$ such that, for any $\Phi\in W^{3,\infty}(\R^d;\R^d)$ such that $\langle \Phi,\nu\rangle \in \{1,\cos(\cdot),\sin(\cdot)\}^\perp$, where $\perp$ denotes the orthogonal for the $L^2(\partial \B)$ inner product,
\[ \langle d^2\mathcal F(\B_2){\Phi},{\Phi}\rangle \geq c_0\Vert \langle \Phi,\nu\rangle\Vert_{W^{\frac12,2}(\partial \B_2)}^2\] where $\nu$ is the normal vector on $\partial \B_2$.
\end{enumerate}

\end{theorem}
This theorem is to be expected as it is a generalisation of the classical Faber-Krahn inequality to the case of the Stokes operator. The condition $\langle \Phi,\nu\rangle\in\{1,\cos,\sin\}^\perp$ reflects the dilation and translation invariance of the functional $\mathcal F$.

\begin{remark}
The differentiability of $\mathcal F$ at $\B_2$ is immediate if $\lambda_1(\B_2)$ is a simple eigenvalue. We refer to \cite[Chapter 5]{zbMATH06838450}, or to the implicit function theorem of Mignot, Murat \& Puel \cite{zbMATH03656460} (for its application in the context of shape derivatives, see \emph{e.g.} \cite{DambrineKateb}).
\end{remark}

\begin{remark} \label{rk:dim1913}
{If a minimiser $\Omega^*$ of \eqref{Eq:Pv} in dimension 2 is smooth and simply connected, then $\Omega^*$ is necessarily a disc.  The proof is due to N.B. Willms and H.F. Weinberger. They did not publish it, but one can find it in the paper of B. Kawohl in \cite{kawohl1998optimal}. We also refer to \cite[Section~11.3.4]{Henrot_2006}. This result is proved by establishing an equivalence between the minimisation of the first eigenvalue of the Stokes-Dirichlet operator and the minimisation of the solution of the buckling problem. Detailed explanations of the buckling problem and the equivalence can be found in section \ref{Se:Biblio}. Once this equivalence is established, the proof that a regular and simply connected minimiser is a disc results from the study of an overdetermined problem.}

{Finally, let us point to the recent article \cite{franzina2024overdetermined}, which was brought to our attention during the revision of this paper: an interesting discussion is given, linking several overdetermined problems including the one mentioned above. }
\end{remark}

\paragraph{The Faber-Krahn inequality is not true for the Stokes operator in $\R^3$.}
A much more surprising result is that \emph{the Faber-Krahn inequality is not true for the Stokes operator in $\R^3$.} To the best of our knowledge, this is the first published result of this kind, where the dimension has an influence on the optimality of the unit ball for the lowest eigenvalue of a differential operator.
\begin{remark}
Shortly after uploading a preprint of this paper, W. Gerner has informed us that similar (unpublished) results \cite{GernerCommunication} had been obtained for some related spectral optimisation problems in the context of Riemaniann geometry, where the ball is optimal in the two-dimensional case, but not in the three-dimensional one. In the euclidean setting, the examples he gave were of a Stokes-like operator, with the additional boundary condition $\langle\curl(\u),\nu\rangle=0$ on $\partial \O$, as well as the requirement that $\u$ be orthogonal to harmonic fields (that is, $\div$ and $\curl$ free fields).  
\end{remark}

More precisely, we have:

\begin{theorem}\label{Th:NonOptimality3d}
The first eigenvalue $\lambda_1(\B_3)$ has multiplicity 3. $\mathcal F$ is semi-differentiable {at $\B_3$}, but does not satisfy the first-order optimality conditions \eqref{Eq:FirstOrderOC}: there exists a vector field {$\Phi\in W^{3,\infty}(\R^3;\R^3)$} such that 
\[ \langle \partial \mathcal F(\B_3),\Phi\rangle<0.\] In particular, $\B_3$ does not solve \eqref{Eq:Pv}.
\end{theorem}
\paragraph{Necessary optimality conditions in $\R^3$.}
In fact, Theorem \ref{Th:NonOptimality3d}, which we singled out, is an easy consequence of the following optimality conditions:
\begin{theorem}\label{Th:Necessary}
Let $\O^*$ be a solution of \eqref{Eq:Pv} {with $d=3$}. 
 If $\O^*$ has a $\mathscr C^{2,\alpha}$ boundary with $\alpha\in (0,1)$ then $\lambda_1(\O^*)$ is a simple eigenvalue. Furthermore, if $\u$ is an associated first eigenfunction, then $\Vert (\n \u)\nu\Vert$ is constant on $\partial \O$.
\end{theorem}

\begin{remark}
The eigenvalues of the Stokes operator are generically simple with respect to the domain; we refer to Ortega \& Zuazua and Chitour, Kateb \& Long \cite{zbMATH06524158,zbMATH01700806}. Regarding the multiplicity of eigenvalues of the Stokes operator let us also mention the recent work of Falocchi \& Gazzola \cite{Falocchi_2021,GazzolaFalocchi2}; it deals with Navier boundary conditions rather than with Dirichlet ones. 
Nevertheless, for an optimal domain (for an eigenvalue other than the first one), multiplicity is very often expected 
\cite[chapter 11]{bookHedG}.  \end{remark}
 \begin{remark}
 It is actually easy to prove that, at an optimiser, $\lambda_1(\O^*)$ has multiplicity at most 2. This will follow from considering the semi-differential of the eigenvalue. However, in order to lower the possible multiplicity to 1, we draw inspiration from the work of  Gerner \cite[Proof of Theorem 2]{Gerner_2023}.
 \end{remark}

Although it is highly unlikely that these necessary optimality conditions will provide a full characterisation of the optimiser, the Poincar\'e-Hopf theorem (also known as the hairy ball theorem, see \cite[Theorem 34.1]{zbMATH03320209} and  \cite[Theorem 39.7]{zbMATH01440270}) easily implies the following corollary:
\begin{corollary}\label{Co:HairyBall}
Assume $\O^*$ has a $\mathscr C^{2,\alpha}$ boundary with $\alpha\in (0,1)$. If $\Sigma$ is a connected component of $\partial \O^*$, $\Sigma$ has Euler characteristic 0, and is thus homeomorphic to a torus.\end{corollary}
Indeed, it suffices to observe that, as for any eigenfunction $\u$ in $\O$, as $\u=0$ in $\partial\O$ and $\n\cdot\u=0$ in $\O$, the vector field $(\n \u)\nu$ is tangential to $\partial \O$ (see Remark \ref{Re:Tangential}). Thus, Theorem \ref{Th:Necessary} yields the applicability of the hairy ball theorem. It is notable that this line of reasoning was previously used in \cite{enciso2022optimal,zbMATH07697301}. These results are in good agreement with the numerical simulations of \cite{li2023shape}, which sets out to study numerical approximations of spectral optimisation problems for the Stokes operator.

\begin{remark}\label{Re:Tangential}[A consequence of choosing Dirichlet conditions]
An elementary yet important observation that will be used several times in the paper is the following fact: for any $\mathscr C^2$ domain $\O$, if $\u_\O$ is a first eigenfunction of the Dirichlet-Stokes operator, then the incompressibility condition and the boundary conditions imply that 
\[ \langle (\n \u_\O)\nu,\nu\rangle=0\quad \text{ on }\partial \O.\]
Indeed, as $\u_\O=0$ on $\partial \O$, we have, for any $1\leq i\leq d$, 
\[ \n u_i=\frac{\partial u_i}{\partial \nu}\nu,\] whence, for any $1\leq i,j\leq d$, 
\[ \frac{\partial u_i}{\partial x_j}=\frac{\partial u_i}{\partial \nu}{\nu_j}.\] 
Consequently, one has 
\begin{eqnarray*}
 \langle (\n \u_\O)\nu,\nu\rangle&=&\sum_{i=1}^d \nu_i\left(\sum_{j=1}^d \frac{\partial u_i}{\partial x_j}\nu_j\right)
 =\sum_{i=1}^d \frac{\partial u_i}{\partial \nu}\nu_i\left(\sum_{j=1}^d \nu_j^2\right)\\
 &=& \sum_{i=1}^d \frac{\partial u_i}{\partial \nu}\nu_i=\sum_{i=1}^d \frac{\partial u_i}{\partial x_i}= \n \cdot \u =0
\end{eqnarray*}
a.e. on $\partial\O$.

\end{remark}

\subsection{Bibliographical references}\label{Se:Biblio}
The problem under consideration in this article fits in several active lines of research. Let us describe some of the main ones.

\paragraph{Spectral optimisation problems for scalar operators.}
Due to the wealth of information they provide on the interplay between the geometry and the analytic properties (\emph{i.e.} regarding functions defined on the domain)  of domains, spectral optimisation problems have become a tenet of applied mathematics. The basic question in spectral optimisation is the following: given a certain spectral quantity, defined through a differential operator on a domain, what is the domain minimising or maximising this spectral quantity? 
It would be pointless to try and give an exhaustive bibliographical account of the developments of the field, but let us highlight some of the key contributions. The seminal works of Faber \cite{faber1923beweis} and Krahn \cite{Krahn_1925}, investigating the minimisation of the first eigenvalue of the (scalar) Dirichlet-Laplacian bolstered the development of new approaches blending geometric and analytic tools. Typical questions include the existence of optimisers, their geometry and their stability. Regarding the existence of optimisers, when no direct comparison principles can yield the explicit description of an optimiser, the first general theorem is due to Buttazzo \& Dal Maso \cite{Buttazzo_1993}, when additional box constraints are enforced on the set of admissible domains. Bucur \cite{zbMATH01454740} developed a framework designed to handle the unbounded case. This was done using the concentration compactness-principle of Lions \cite{LionsCC}, and later used by Bucur \& Varchon \cite{BucurVarchon} and Bucur \& Freitas \cite{BucurFreitas} to derive existence results for the optimisation of eigenvalues of scalar operators. Let us also mention some more general results due, independently, to Bucur and Mazzoleni \& Pratelli \cite{Buc12,MazPra11}. 
We refer to \cite[chapter 2]{bookHedG} as a reference for existence problems. Regarding the geometry of optimisers, these are usually very delicate questions. While symmetrisation techniques-which can be used to derive direct comparison results-can be available for specific problems, there are no general tools that can be employed for generic spectral optimisation problems. Regarding rearrangements, we refer to the monograph \cite{zbMATH05042914}. Another possibility to obtain information regarding the optimal domains is to derive tractable optimality conditions; as the latter are often expressed in terms of an overdetermined boundary value problem having a solution on the optimal domain, results similar to the Serrin theorem  \cite{zbMATH03351993} (see also \cite{zbMATH06880881}) yield the result. For an introduction to the geometric aspects of spectral optimisation we refer to \cite{Henrot_2018}.

\paragraph{Optimality conditions for multiple eigenvalues.}
A tenet of the Stokes problem in dimension 3 is that the first eigenvalue of the Dirichlet-Stokes operator has multiplicity 3 at the ball. While seemingly innocuous this remark actually implies the non-minimality of the ball in $\R^3$. Naturally, when deriving optimality conditions for multiple eigenvalues, one needs to be very careful in handling shape derivatives, as they do not exist. The approach used in this article is based on the apparatus that was set up by Cox \cite{zbMATH00850600}, Chenais \& Rousselet \cite{Rousselet_1990}  and Chatelain \& Choulli \cite{zbMATH01167545} to handle these difficulties. The relevant formulae will be recalled in the course of this article. In a recent paper, Caubet, Dambrine \& Mahadevan \cite{Caubet_2021} developed this semi-differential approach to accomodate the linear elasticity system.

\paragraph{Shape optimisation for vectorial operators.}
The literature devoted to shape optimisation problems for vectorial operators is scarce, but growing. Of particular importance to us in this regard are the different contributions to the study of isoperimetric inequalities for the $\curl$ operator {in $\R^3$}, that is, the optimisation of the first positive eigenvalue of 
\[\begin{cases}
\curl(\u)=\xi \u & \text{ in }\O\,,\\ \langle \u,\nu\rangle=0 & \text{ on }\partial \O.
\end{cases}
\] This problem is of paramount importance in the study of magnetic fields, and has a strong connection \cite{Chandrasekhar_1957} to the Stokes eigenvalue problem with a tangential flow condition, rather than a Dirichlet condition. To the best of our knowledge, the main works on this $\curl$-isoperimetric problems can be found, on the one hand, in the works of Cantarella, DeTurck, Gluck  and Teytel \cite{Cantarella_2000_Bis,Cantarella_2000} and, on the other hand, in the research of Enciso, Gerner \& Peralta-Salas \cite{enciso2022optimal,zbMATH07697301,Gerner_2023}. Let us briefly review their finding: in \cite{Cantarella_2000_Bis,Cantarella_2000}, the 	authors, using explicit computations of eigenfunctions, observe the optimality of the ball among all concentric spherical shells with a given volume constraint, but also conjecture that the minimiser of the $\curl$-isoperimetric problem is a ``spheromak", that is, a sphere where the south and north poles are glued together. In their paper, they also give optimality conditions for the optimiser, which lead to the same type of conclusion as Corollary \ref{Co:HairyBall}. It is interesting to note that their analysis recovers parts of the semi-differentiability results of \cite{zbMATH01167545}. Before we describe the results of \cite{enciso2022optimal,zbMATH07697301,Gerner_2023} observe that Enciso, Gerner \& Peralta-Salas work with vector fields that are not only incompressible, but also orthogonal to harmonic (\emph{i.e.} $\curl$ and $\div$ free) fields. In \cite{enciso2022optimal,zbMATH07697301}, another take on the problem is introduced, and the authors investigate possible symmetries of optimisers. In \cite{zbMATH07697301}, Enciso \& Peralta-Salas investigate whether a solution of the $\curl$-isoperimetric problem can have axial symmetry; their conclusion is no, provided some regularity of the minimiser is assumed \emph{a priori}. Their results hinge on a variational analysis (\emph{\`a la } Hadamard). In \cite{enciso2022optimal} on the other hand, they seek optimisers in the class of convex domains, and they prove that one can not have ``too regular" optimisers.  Gerner, in \cite{Gerner_2023}, went much further in the fine characterisation of possible optimisers, both in the euclidean and riemannian settings . Although the adaptation of some of his results is not immediate, one of the necessary optimality conditions we derive is inspired by \cite[Theorem 2]{Gerner_2023}. Finally, let us mention the recent work of Lamberti \& Zaccaron \cite{zbMATH07394264}, in which, also using a semi-differential approach, the authors investigate the optimal shape of an electromagnetic cavity.

\paragraph{The buckling problem.}
Let us conclude this bibliographical paragraph by mentioning a closely related problem in the two-dimensional case, the buckling problem. Indeed, in dimension 2, assume $\O$ is simply connected, so that we might write the first Dirichlet-Stokes eigenfunction $\u$ as 
\[ \u=\begin{pmatrix}-\partial_y \psi\\ \partial_x \psi\end{pmatrix}\] for some function $\psi$. { 
Plugging this expression into \eqref{Eq:Eigenequation} we deduce that 
\[\begin{cases}
-\partial_{xxy}\psi-\partial_{yyy}\psi-\partial_xp=\lambda_1(\O)\partial_y\psi& \text{ in }\O\,, \\
-\partial_{xxx}\psi-\partial_{yyx}\psi+\partial_yp=\lambda_1(\O)\partial_x\psi&\text{ in }\O.
\end{cases}
\] Differentiating the first equation with respect to $y$, the second with respect to $x$ and summing the two, we deduce that $-\Delta^2\psi=\lambda_1(\O)\Delta\psi$. Furthermore, as $\begin{pmatrix}-\partial_y\psi\\\partial_x\psi\end{pmatrix}=\begin{pmatrix}0\\0\end{pmatrix}$ on $\partial \O$ we deduce that $\psi$ is constant on $\partial \O$, and that $\partial_\nu\psi=0$. Up to adding a constant to $\psi$, we may take $\psi=0$ on $\partial \O$. Overall, $(\psi, \lambda_1(\O))$  solves the fourth order equation
\[ \begin{cases}
-\Delta^2\psi=\lambda_1(\O)\Delta \psi&\text{ in }\O\,, 
\\ \psi=\partial_\nu \psi=0&\text{ on }\partial \O.
\end{cases}\] 
}
In fact, with a bit more work, one can see that $\lambda_1(\O)$ coincides with the eigenvalue
\[ \Lambda(\O):=\min_{v\in W^{2,2}(\O)\cap W^{1,2}_0(\O)}\frac{\int_\O (\Delta v)^2}{\int_\O |\n v|^2}.\]
This is the well-known \emph{buckled plate eigenvalue problem}, leading to the spectral optimisation problem
\begin{equation}\inf_{\Omega \subset \R^2\,, |\O|\leq c}\Lambda(\O).\end{equation} We refer to \cite[Chapter 11]{Henrot_2006} for more details, but let us underline the following aspects of the problem:
first of all, the existence of an optimal domain remained open until the contribution of Ashbaugh \& Bucur \cite{Ashbaugh_2003} and, later, of Stollenwerk \cite{zbMATH07633006,zbMATH06560656}. Second, this problem is related to a long-standing conjecture by P\'oly\`a \& Szeg\"{o} that asserts that the ball is a solution of this variational problem. 
A very active line of research has focused on the computation of shape-derivatives for this problem and, more generally, for polyharmonic problems. Regarding Hadamard type shape derivatives, we refer to \cite{MR2720607}. For applications of this calculus to polyharmonic problems, let us point to the numerous works of Buoso \& Lamberti \cite{MR3571822,MR3182687,MR3153100}. To the best of our knowledge, the second-order derivative of the first eigenvalue at the ball was not known.

\section{Proof of Theorem \ref{Th:Existence}}

We pick a minimising sequence $\{\O_k\}_{k\in \N}$ for \eqref{Eq:Pv}. By scaling invariance of $\mathcal F$ we might assume that $|\O_k|$ does not depend on $k$: there exists $V_0>0$ such that $|\O_k|=V_0$ for any $k\in \N$. We consider the auxiliary problem 
\begin{equation}\label{Eq:PvAuxiliary}
\inf_{\O\in \mathscr O\,, |\O|\leq V_0}\lambda_1(\O).
\end{equation}
Let us consider, for any $k\in \N$, the function $w_k$ defined as 
\[ w_k:=\Vert \u_k\Vert\] where for any $k$ the function $\u_k$ is an\footnote{We do not rule out the possibility of $\O_k$ having a multiple eigenvalue. Although one could easily discard this case by invoking the fact that the first eigenvalue is generically simple, we do not require this.} eigenfunction of $\O_k$. The function $w_k$ is extended by 0 outside of $\O_k$.

By definition and elementary computations, we have 
\begin{equation}\label{Eq:Prprtywk} \int_{\R^d}w_k^2=\int_\O\Vert \u_k\Vert^2\color{black}=1,\quad  \int_{\R^d}|\n w_k|^2\leq \int_{\R^d}\Vert \n \u_k\Vert^2.\end{equation}

 By the concentration-compactness principle of Lions \cite{LionsCC} we know that, up to extracting a subsequence which we still denote $\{w_k\}_{k\in \N}$ with a slight abuse of notations, \color{black} one of the following occurs:
\begin{itemize}
\item[(i)] \textbf{Concentration}: there exists a function $w\in L^2(\R^d;\R_+)$ and a sequence $\{y_k\}_{k\in \N}\in (\R^d)^\N$ such that 
$w_k(\cdot+y_k)\xrightarrow[k\to \infty]{} w$ in $L^2(\R^d)$ and weakly in $W^{1,2}(\R^d)$. 
\item[(ii)]\label{Dichotomy} \textbf{Dichotomy}:  
There exists $\alpha_1\in (0;1)$, $\{y_k\}_{k\in \N}\in (\R^d)^\N$, two sequences $\{R_k,R_k'\}_{k\in \N}$ such that 
\[ R_k-R_k'\xrightarrow[k\to +\infty]{} +\infty,\quad  R_k\,, R_k'\xrightarrow[k\to +\infty]{} +\infty\]
and such that 
\[ 
\int_{\B(y_k,R_k)}w_k^2\xrightarrow[k\to +\infty]{} \alpha_1,\quad  \int_{\B(y_k,R_k')^c}w_k^2\xrightarrow[k\to +\infty]{} 1-\alpha_1
\]
and
\[ \underset{k\to\infty}{\lim\inf}\left(\int_{\R^d}|\n w_k|^2- \int_{\B(y_k,R_k)}|\n w_k|^2- \int_{\B(y_k,R_k')^c}| \n w_k|^2\right)\geq 0.\]
\item[(iii)] \textbf{Vanishing}: for any $r>0$, 
\[ \lim_{k\to \infty}\sup_{y\in \R^d}\int_{\B(y,r)}w_k^2=0.\]
\end{itemize}
In this first step of the proof, we rule out vanishing and dichotomy.
\paragraph{Vanishing does not occur.} Let us exclude vanishing. By  the exact same arguments as in \cite[Lemma 3.3, Proof of Theorem 3.2]{BucurVarchon} and \eqref{Eq:Prprtywk}, if the sequence $\{w_k\}_{k\in \N}$ vanishes, we have\[ \underset{k\to +\infty}{\lim\sup}\, \lambda_1(\Omega_k)=+\infty.\]  This contradicts the fact that the sequence is minimising.
 \paragraph{Dichotomy does not occur.}
 Argue by contradiction and assume dichotomy holds. In that case, define 
 \[ \eta_k:=\frac{R_k'-R_k}4.\] Introduce $\{\psi_{k,1}\,, \psi_{k,2}\}_{k\in \N}$ as two non-negative, radially (with respect to $y_k$) symmetric,  
 non-increasing functions such that 
 \[ \psi_{k,1}(x)=\begin{cases}1\text{ in }\B(y_k,R_k+\eta_k\color{black})\,, 
 \\0 \text{ in }\B(y_k,R_k+2\eta_k)^c,\end{cases}\and  \psi_{k,2}(x)=\begin{cases} 1\text{ in }\B(y_k,R_k'-\eta_k)^c\,, \\ 0\text{ in }\B(y_k,R_k'-2\eta_k)\end{cases}\]and 
 \[ 
 \Vert \n \psi_{k,i}\Vert_{L^\infty}\leq \frac1{\eta_k},\quad  k\in \N,\quad  i\in \{1,2\}.
 \]
 Now, let, for any $k\in \N$ and any $i\in \{1,2\}$
 \[ \v_{k,i}:=\psi_{k,i}\u_k.\] Then,  taking into account that $\Vert \n\psi_{k,1}\Vert_{L^\infty(\R^d)}\,,  \Vert \n\psi_{k,1}\Vert_{L^\infty(\R^d)}\leq 1/\eta_k$, we can proceed exactly as in \cite[Lemma III.1]{LionsCC} to obtain \color{black}
 \[ 
 \int_{\R^d}\Vert \v_{k,1}\Vert^2\underset{k\to\infty}\rightarrow \alpha_1,\quad  \int_{\R^d}\Vert \v_{k,2}\Vert^2\underset{k\to\infty}\rightarrow 1-\alpha_1
 \] 
 and 
 \[ 
 \underset{k\to\infty}{\lim\inf}\left(\int_{\R^d}\Vert \n\u_k\Vert^2-\int_{\R^d}\Vert \n\v_{k,1}\Vert^2-\int_{\R^d}\Vert \n\v_{k,2}\Vert^2\right)\geq 0.
 \]
 Up to iterating that construction, we can assume that $\{\Vert \v_{k,1}\Vert\}_{k\in \N}$ is in a situation of concentration. In particular, there exists $\v_1$ such that, up to a translation by a vector $z_k\in \R^d$ (which we take equal to 0 without loss of generality), 
 \[ \v_{k,1}\xrightarrow[k\to +\infty]{} \v_1\text{ strongly in $L^2(\R^d)$, weakly in $W^{1,2}(\R^d)^3$}.\]
Let us now observe that 
\[ \nabla\cdot\v_1=0\] in the sense of distributions.
To prove this fact, let $\phi\in \mathscr C^\infty_c(\R^d)$ and consider the quantity
\[ \int_{\R^d} \langle \v_{k,1},\n \phi\rangle.\]
As $\phi$ is compactly supported, $\mathrm{supp}(\phi)\cap\{\n\cdot\v_{k,1}\neq 0\}=\emptyset$ for $k$ large enough. Indeed, should this not be the case, as $\n \cdot \v_{k,1}=\n\cdot\u_k=0$   in $\B(y_k,R_k+\eta_k)$, we conclude that there exists $x\in \mathrm{supp}(\phi)\cap\B(y_k,R_k+\eta_k)^c$. Letting $\mathrm{dist}(E,F)$ be the Hausdorff distance between two closed subset $E\,, F\subset \R^d$, this implies
\[ \mathrm{dist}\left(\mathrm{supp}(\phi),\B(y_k,R_k)\right)\geq \eta_k\xrightarrow[k\to +\infty]{} \infty.\] In particular, we deduce that
\[\v_{k,1}\underset{k\to\infty}\rightarrow 0\text{ in }L^2_{\mathrm{loc}}(\R^d),
\]which contradicts the strong $L^2$ convergence of $\{\v_{k,1}\}_{k\in \N}$, thereby leading to a contradiction. Thus, 
\[ \int_{\R^d}\langle \v_{k,1},\n\phi\rangle=0\] for any $k$ large enough. Passing to the limit provides the result.
\color{black}

Now, we know that 
\[\lambda^*:=\inf_{\O\in \mathscr O\,, |\O|\leq V_0}\lambda_1(\O)\geq \min_{i=1,2}\underset{k\to \infty}{\lim\inf}\left\{\frac{\int_{\R^d}\Vert \n\v_{k,i}\Vert^2}{\int_{\R^d}\Vert \v_{k,i}\Vert^2}\right\}.\] If this minimum is reached for $i=1$, then $\O^*:=\{\v_1\neq 0\}$ is an optimal domain. Else, we apply the procedure once more (as in \cite{BucurVarchon}) to $\{\v_{k,2}\}_{k\in \N}$. Either this sequence is in a situation of compactness, in which case we are done, or it is once again in a situation of dichotomy, giving rise to two new sequences $\{\v_{k,2}^1\,, \v_{k,2}^2\}_{k\in \N}$, with $\{\v_{k,2}^1\}_{k\in \N}$ in a situation of compactness. We then iterate the construction on $\v_{k,2}^2$ if necessary. Either this process stops, in which case existence follows, or we obtain  a decreasing sequence $\{\alpha_j\}_{j\in \N\,, j\geq 2}$ such that 
\[ \sum_{j=1}^{+\infty} \alpha_j\leq 1\] and we have a sequence $\{\v_{k,2}^j\}_{k\in \N\,, j\leq J_k }$ where $J_k\xrightarrow[k\to +\infty]{} \infty$ such that for any $j$,  $\{\v_{k,2}^j\}_{k\in \N}$ is in a situation of concentration, and $\{\{\v_{k,2}^{J_k}\neq 0\}\}_{k\in \N}$  is a minimising sequence. 

Observe that if $\sum_{j=1}^\infty\alpha_j<1$, then vanishing holds for $\{\v_{k,2}\}_{k\in \N}$, which leads to a contradiction, so that we can work under the assumption that 
\[ \sum_{j=1}^\infty \alpha_j=1.\] Consequently, we deduce that

\[ \lambda^*\geq \underset{k\to \infty}{\lim\inf}\frac{\int_{\R^d}\Vert \n \v_{k,1}\Vert^2+\sum_{j=1}^{J_k}\int_{\R^d}\Vert \n \v_{k,2}^j\Vert^2}{\int_{\R^d}\Vert \v_{k,1}\Vert^2+\sum_{j=1}^{J_k}\int_{\R^d}\Vert  \v_{k,2}^j\Vert^2}.\]
Let 
\[ \alpha_{k,1}:=\int_{\R^d}\Vert  \v_{k,1}^j\Vert^2\text{ which satisfies }\alpha_{k,1}\underset{k\to\infty}\rightarrow \alpha_1>0\] and 
\[ \tilde \alpha_{k}:=\sum_{j=1}^{J_k}\int_{\R^d}\Vert  \v_{k,2}^j\Vert^2\text{ which satisfies }\tilde \alpha_{k,1}\underset{k\to\infty}\rightarrow 1-\alpha_1.\] 
Thus
\[\lambda^*\geq \underset{k\to \infty}{\lim\inf}\frac{\int_{\R^d}\Vert \n \v_{k,1}\Vert^2}{\alpha_{k,1}+\tilde\alpha_k}+\underset{k\to \infty}{\lim\inf}\frac{\sum_{j=1}^{J_k}\int_{\R^d}\Vert\n  \v_{k,2}^j\Vert^2}{\alpha_{k,1}+\tilde\alpha_k}.\]
Now, since $\alpha_{k,1}+\tilde\alpha_k\underset{k\to\infty}\rightarrow 1\,, \alpha_{k,1}\underset{k\to\infty}\rightarrow \alpha_1$, the fact that we have by construction 
\[ \underset{k\to \infty}{\lim\inf}\frac{\int_{\R^d}\Vert \n \v_{k,1}\Vert^2}{\alpha_1}>\lambda^*\]implies
\[ \underset{k\to \infty}{\lim\inf}\frac{\int_{\R^d}\Vert \n \v_{k,1}\Vert^2}{\alpha_{k,1}+\tilde\alpha_k}>\alpha_1\lambda^*.\]
Likewise, 
\[\underset{k\to \infty}{\lim\inf}\frac{\sum_{j=1}^{J_k}\int_{\R^d}\Vert  \n\v_{k,2}^j\Vert^2}{\alpha_{k,1}+\tilde\alpha_k}\geq (1-\alpha_1)\lambda^*.\] Consequently
\[ \lambda^*\geq \underset{k\to \infty}{\lim\inf}\frac{\int_{\R^d}\Vert \n \v_{k,1}\Vert^2+\sum_{j=1}^{J_k}\int_{\R^d}\Vert \n \v_{k,2}^j\Vert^2}{\int_{\R^d}\Vert \v_{k,1}\Vert^2+\sum_{j=1}^{J_k}\int_{\R^d}\Vert  \v_{k,2}^j\Vert^2}>\lambda^*,\] a contradiction. The result follows.

We deduce that we have concentration of the sequence, so that 
\[ \v_k\underset{k\to\infty}\rightarrow \v\text{ in $L^2(\R^d)$}.\]
Define 
\[ \O^*:=\{\v\neq 0\}.\] 
$\O^*$ is a quasi-open set, and $|\O^*|\leq V_0$, as, up to a subsequence, $\mathds 1_{\O^*}\leq \underset{k\to\infty}{\lim\inf} \, \mathds 1_{\{\v_k\neq 0\}}$. Furthermore, 
\[ \lambda_1(\O^*)\leq \frac{\int_{\O^*}\Vert \n \v\Vert^2}{\int_{\R^d}\Vert \v\Vert^2}\leq \underset{k\to\infty}{\lim\inf}\frac{\int_{\O^*}\Vert \n \v_k\Vert^2}{\int_{\R^d}\Vert \v_k\Vert^2}=\lambda^*,\] where we used the weak $W^{1,2}(\R^d;\R^d)$ convergence of the sequence $\{\v_k\}_{k\in \N}$.

It follows that $\Omega^*$ solves ~\eqref{Eq:PvAuxiliary}.

\color{black}

\section{Proof of Theorem~\ref{Th:Optimality2d} }
Throughout our analysis, we will be using some basic facts about the Stokes eigenfunctions in $\B_2$, which we thus first recall.
\subsection{Preliminaries}
\paragraph{The first eigenpair in $\B_2$.}
We let $E_1(\B_2)$ be the first eigenspace associated with $\lambda_1(\B_2)$. We begin with the following fact (see Remark \ref{Re:Buckling}):
\begin{lemma}\label{Le:Multiplicity2d}
The eigenvalue $\lambda_1(\B_2)$ is simple: $\mathrm{dim}(E_1(\B_2))=1$. Furthermore, 
\[ \lambda_1(\B_2)=j^2_{1,1},
\] where $j_{1,1}$ is the first positive root of $J_1$, the Bessel function of the first kind of order 1.
 Finally,
$E_1(\B_2)$ is spanned by the eigenfunction 
{\[ \u_{\B_2}:=
\frac{J_1(j_{1,1} r) }{ \sqrt{\pi}   \vert J_0(j_{1,1})\vert }  \begin{pmatrix} -\sin\theta \\ \cos\theta \end{pmatrix} ,
\]}
{normalised in $L^2(\B_2)$}, and the {associated} pressure is constant.
\end{lemma}

{\begin{remark}[A useful computation]\label{rk:usefulComp}We single out the following computations, which will prove useful when studying the local optimality of the ball.
Let us set $c=1/(\sqrt{\pi}|J_0(j_{1,1})|)$, $r=\sqrt{x^2+y^2}$ and $f(r)=J_1(j_{1,1}r)/r$. For $r>0$, one easily computes
\[
\nabla \u_{\B_2}=c\begin{pmatrix}
-\frac{xy}{r}f'(r) & -f(r)-\frac{y^2}{r}f'(r)\\
f(r)+\frac{x^2}{r}f'(r) & \frac{xy}{r}f'(r)
\end{pmatrix}
\]
so that
\[
\nabla \u_{\B_2}=cf'(1)\begin{pmatrix}
-\cos \theta\sin\theta & -\sin^2\theta\\
\cos^2(\theta) & \cos\theta\sin\theta
\end{pmatrix}\qquad \text{on }\partial\O.
\]
In particular, observe that
\[ (\n \u_{\B_2})\nu= cf'(1)\begin{pmatrix}-\sin(\theta)\\ \cos(\theta)\end{pmatrix}.\]
\end{remark}}

While we single out this result for later reference, the following spectral decomposition of the Dirichlet-Stokes operator in the ball is well-known:

\begin{lemma}\label{Le:Eigenbasis2d}
A Hilbert basis of eigenfunctions of the Dirichlet-Stokes operator is given, in polar coordinates, by 
{
\begin{equation}\label{Stokes_eig0}
\phi_{0,k}(r,\theta) = \frac{-J_0'(\sqrt{\lambda_{0,k}} r) }{ \sqrt{\pi} \vert J_0(\sqrt{\lambda_{0,k}})\vert }  \begin{pmatrix} -\sin\theta \\ \cos\theta \end{pmatrix} ,
\end{equation}}
and
{\begin{equation}\label{Stokes_eig1}
\begin{split}
\phi_{j,k,m}(r,\theta) \ =\ &
\frac{J_j(\sqrt{\lambda_{j,k}} r)-J_j(\sqrt{\lambda_{j,k}})r^j}{\sqrt{\lambda_{j,k}}\vert J_j(\sqrt{\lambda_{j,k}})\vert r} j (-1)^{m+1} Y_{j,m}(\theta) \begin{pmatrix} \cos\theta \\ \sin\theta \end{pmatrix} \\
& +
\frac{-\sqrt{\lambda_{j,k}}J_j'(\sqrt{\lambda_{j,k}} r) + j J_j(\sqrt{\lambda_{j,k}}) r^{j-1} }{ \sqrt{\lambda_{j,k}} \vert J_j(\sqrt{\lambda_{j,k}})\vert } Y_{j,m+1}(\theta) \begin{pmatrix} -\sin\theta \\ \cos\theta \end{pmatrix}
\end{split}
\end{equation}}
for $j\in\Z^*$, $k\in\N^*$ and $m=1,2$.

Here, $(r,\theta)$ are the usual polar coordinates (see \cite{Kelliher,LeeRummler}).
The functions $Y_{j,m}(\theta)$ are defined by $Y_{j,1}(\theta)=\frac{1}{\sqrt{\pi}}\cos(j\theta)$ and $Y_{j,2}(\theta)=\frac{1}{\sqrt{\pi}}\sin(j\theta)$, with the agreement that $Y_{j,3}=Y_{j,1}$, and $J_j$ is the Bessel function of the first kind of order $j$. Denoting by $j_{j,k}>0$ the $k^\textrm{th}$ positive zero of $J_{j}$, the eigenvalues of  the Dirichlet-Stokes operator are the doubly indexed sequence $(\lambda_{j,k})_{j\in\Z,k\in\N^*}$, where $\lambda_{j,k}=j_{\vert j\vert+1,k}^{2}$ is of multiplicity $1$ if $j=0$, and $2$ if $j\neq 0$.

\end{lemma}

\begin{remark}\label{Re:Buckling}[Regarding Lemmata \ref{Le:Multiplicity2d}-\ref{Le:Eigenbasis2d}]
Lemmata  \ref{Le:Multiplicity2d}-\ref{Le:Eigenbasis2d} are essentially contained in \cite{Kelliher,LeeRummler}, which provide explicit computations of the eigen-elements of the Dirichlet-Stokes operator in special geometries. While we refer to these articles for detailed computations, let us briefly indicate how they might be derived in a straightforward manner: consider $(\u_k,\lambda_k(\B_2))$ an eigenpair of the Dirichlet-Stokes operator. As $\nabla\cdot \u_k=0$, there exists a scalar function $\psi_k$ such that
\[ \u_k=\curl(\psi_k):=\begin{pmatrix}-\partial_y\psi_k\\\partial_x\psi_k\end{pmatrix}\] and, taking the (2 dimensional) $\curl$ of \eqref{Eq:Eigenequation}, it appears that $\psi_k$ solves
\begin{equation}\begin{cases}
-\Delta^2\psi_k=\lambda_k(\B_2)\Delta \psi_k&\text{ in }\B_2\,, 
\\ \psi_k=\partial_\nu\psi_k=0&\text{ on }\partial \B_2.\end{cases}\end{equation}
In other words, $\psi_k$ is an eigenvalue of the aforementioned buckling problem (see section \ref{Se:Biblio}). In other words, knowing the eigenspaces of the buckling problem leads to determining the eigenspaces of the Dirichlet-Stokes operator. However, the eigenspaces of the buckling operator can be easily computed in the usual radial coordinates. It should be observed that the three-dimensional case, although it follows a similar pattern (boiling it down to a scalar operator) is much more involved. We refer to Appendix \ref{Ap:Saks}.
\end{remark}

\paragraph{Basic results about the shape differentiability of eigenfunctions and eigenvalues.}
Since $\lambda_1(\B_2)$ is simple, the following differentiability result follows from the implicit function theorem of Mignot, Murat \& Puel \cite{zbMATH03656460}:
\begin{lemma}\label{Le:Differentiability2D}
Let \[ \mathscr X:=\{\Phi\in W^{3,\infty}(\R^2;\R^2)\,, \Phi \text{ compactly supported}\}.\] For any $\mathscr C^2$ domain $\O$ such that $\lambda_1(\O)$ is simple, the eigenvalue mapping $\O\mapsto \lambda_1(\O)$ is twice differentiable at $\Omega$ in the following sense: for any $\Phi\in \mathscr X$, the map $f_\Phi:t\mapsto \lambda_1\left((\mathrm{Id}+t\Phi)\O\right)$ is twice differentiable at $t=0$. We will use the notations 
\[   \langle d\lambda_1(\O),\Phi\rangle :=f_\Phi'(0)\,,    \langle d^2\lambda_1(\O)\Phi,\Phi\rangle :=f_\Phi''(0).\]
Similarly, the mapping $\O\mapsto \u_\O$ is twice differentiable  at $\Omega$, where $\u_\O$ is the first Dirichlet-Stokes normalized eigenfunction of $\O$, in the sense that the mapping $g_\Phi:t\mapsto \u_{(\mathrm{Id}+t\Phi)\Omega}$ is twice differentiable at $t=0$. We let $\u_\Phi'$ be its derivative at $t=0$.
\end{lemma}
For a detailed introduction to the Hadamard shape calculus, we refer to \cite[Chapter 5]{zbMATH06838450} and to \cite{MR2720607}. To proceed with the proof of Theorem \ref{Th:Optimality2d}, we need tractable expressions for the first and second-order shape derivatives of the eigenvalue at a domain $\O$.

\begin{proposition}\label{Pr:DifferentiabilityFormulae}
For any $\mathscr C^2$ domain $\O$ such that $\lambda_1(\O)$ is simple, let $\u_\O$ be its associated first eigenfunction. For any $\Phi\in \mathscr X$, the shape derivative $\u_\Phi'$ solves 
\begin{equation} \label{Pb:der}
\begin{cases}-\Delta {\u_\Phi'}+\nabla p'=\lambda_1(\O){\u_\Phi'}+   \langle d\lambda_1(\O),\Phi\rangle  {\u_\O} & \textrm{in }\O\\
\nabla\cdot {\u_\Phi'}=0 & \textrm{in }\O\\
{\u_\Phi'}=-\nabla {\u_\O} \nu{\langle \Phi,\nu\rangle} & \textrm{on }\partial \O\,, 
\\ \int_{\O} \langle \u_\O,\u_\Phi'\rangle=0.
\end{cases}
\end{equation}
The first order derivative of $\lambda_1$ is 
\begin{equation}\label{Eq:D1Lambda}
 \langle d\lambda_1(\O),\Phi\rangle=-\int_{\partial \O}\Vert( \n \u_\O)\nu\Vert^2\langle \Phi,\nu\rangle.\end{equation} If, in addition, the vector field $\Phi$ is normal to $\partial \O$, the second-order shape derivative of $\lambda_1$ at $\O$ is given by \begin{equation}\label{Eq:D2Lambda}
   \langle d^2\lambda_1(\O)\Phi,\Phi\rangle =2\int_{\partial \O}\langle \u_\Phi',(\n \u_\Phi')\nu\rangle+\int_{\partial \O} H \Vert (\n \u_\O)\nu\Vert^2\langle \Phi,\nu\rangle^2-2\int_{\partial \O}\langle (\n \u_\O)\nu,\n p_\O\rangle
\end{equation} where $H$ is the mean curvature of $\partial \O$ and $p_\O$ is the pressure field associated with $\u_\O$.
\end{proposition}
\begin{proof}[Proof of Proposition \ref{Pr:DifferentiabilityFormulae}]
That $\u_\Phi'$ solves \eqref{Pb:der} is a standard consequence of general formulae for the shape differentiation of Dirichlet boundary value problem, and we refer to \cite[Chapter 5]{zbMATH06838450} for the detailed computations. {To derive the expression of $\langle d\lambda_1(\O),\Phi\rangle$, multiply \eqref{Pb:der} by $u_\O$ to obtain:
\begin{align*}
\langle d\lambda_1(\O),\Phi\rangle&=\langle d\lambda_1(\O),\Phi\rangle\int_\O\Vert \u_\O\Vert^2\quad \left(\text{ since }\int_\O \Vert \u_\O\Vert^2=1\right)
\\&=\int_\O  \n \u_\Phi'\colon\n\u_\O -\underbrace{\int_{\partial\O}\langle (\n\u_\Phi')\nu,\n\u_\O\rangle}_{=0\text{ since }\u_\O=0\text{ on }\partial \O}
+\underbrace{\int_\O\langle \n p',\u_\O\rangle}_{=0\text{ since }\n\cdot\u_\O=0}-\lambda_1(\O)\underbrace{\int_\O \langle \u_\Phi',\u_\O\rangle}_{=0\text{ since }\int_\O \langle \u_\O\,, \u_\Phi'\rangle=0}
\\&=-\int_\O\langle \Delta \u_\O\,, \u_\Phi'\rangle+\int_{\partial \O}\langle (\n\u_\O)\nu,\u_\Phi'\rangle
\\&=-\underbrace{\int_\O\langle \n p_\O,\u_\Phi'\rangle}_{=0\text{ since }\n\cdot\u_\Phi'=0}+\lambda_1(\O)\underbrace{\int_\O \langle \u_\Phi',\u_\O\rangle}_{=0\text{ since }\int_\O \langle \u_\O\,, \u_\Phi'\rangle=0}+\int_{\partial \O}\langle (\n\u_\O)\nu,\u_\Phi'\rangle
\\&=-\int_{\partial \O} \Vert (\n\u_\O)\nu\Vert^2\langle \Phi,\nu\rangle.
\end{align*}} 

To obtain the expression for $   \langle d^2\lambda_1(\O)\Phi,\Phi\rangle $ for normal vector fields, it suffices to apply the Hadamard formula for integrals on variable boundaries \cite[Proposition 5.4.18]{zbMATH06838450}. This yields
\begin{align*}
   \langle d^2\lambda_1(\O)\Phi,\Phi\rangle &=-2\int_{\partial \O}\langle (\n \u_\Phi')\nu,(\n\u_\O)\nu\rangle \langle \Phi,\nu\rangle
\\&-\int_{\partial \O}\left(H\Vert (\n \u_\O)\nu\Vert^2+\frac{\partial ||(\n \u_\O)\nu||^2}{\partial \nu}\right)\langle \Phi,\nu\rangle^2.
\end{align*}
However, recall that, introducing $\Delta_\tau$ the tangential laplacian on $\partial \O$, we have \[ \Delta\u=\Delta_\tau \u+H(\n \u)\nu+\frac{\partial ((\n \u)\nu)}{\partial \nu}.\]{
In particular, this implies, taking into account that $\Delta_\tau\u_\O=0\,, \u_\O=0$ on $\partial\O$,
\[\frac{\partial ((\n \u_\O)\nu)}{\partial \nu}=\Delta \u_\O-\Delta_\tau\u_\O-H(\n\u_\O)\nu=-\lambda_1(\O)\u_\O-\n p_\O-H(\n\u_\O)\nu=-\n p_\O-H(\n\u_\O)\nu.\]
Consequently, 
\begin{equation}\label{Eq:Ono}
\frac{\partial \Vert (\n \u_\O)\nu\Vert^2}{\partial\nu}=2\left\langle (\n\u_\O)\nu,\frac{\partial( (\n\u_\O)\nu)}{\partial\nu}\right\rangle=-2\langle \n p_\O\,, (\n\u_\O)\nu\rangle-2H\Vert (\n\u_\O)\nu\Vert^2,
\end{equation}
so that 
\[ H\Vert (\n \u_\O)\nu\Vert^2+\frac{\partial ||(\n \u_\O)\nu||^2}{\partial \nu}=-H\Vert (\n\u_\O)\nu\Vert^2-2\langle \n p_\O\,, (\n\u_\O)\nu\rangle\]and, finally, since $(\n\u_\O)\nu\langle\Phi,\nu\rangle=-\u_\Phi'$ on $\partial \O$,
\begin{align*}
   \langle d^2\lambda_1(\O)\Phi,\Phi\rangle &=-2\int_{\partial \O}\langle (\n \u_\Phi')\nu,(\n\u_\O)\nu\rangle \langle \Phi,\nu\rangle
-\int_{\partial \O}\left(H\Vert (\n \u_\O)\nu\Vert^2+\frac{\partial ||(\n \u_\O)\nu||^2}{\partial \nu}\right)\langle \Phi,\nu\rangle^2
\\&=2\int_{\partial \O}\langle \u_\Phi',(\n \u_\Phi')\nu\rangle+\int_{\partial \O} H \Vert (\n \u_\O)\nu\Vert^2\langle \Phi,\nu\rangle^2-2\int_{\partial \O}\langle (\n \u_\O)\nu,\n p_\O\rangle
\langle \Phi,\nu\rangle^2
\end{align*}}
as claimed.
\end{proof}

\subsection{The ball is a critical point for Problem \eqref{Eq:Pv} (in dimension two)}\label{sec:ballcritpt}
We first prove that $\B_2$ is a critical point for $\mathcal F$ (recall that $\mathcal F(\O)=|\O|\lambda_1(\O)$). 

\begin{proposition}\label{prop:B2critical}
The ball $\B_2$ is a critical point for the problem \eqref{Eq:Pv}: for any $\Phi\in \mathscr X$, there holds
\[\langle d\mathcal F(\B_2),\Phi\rangle=0.\]
\end{proposition}
\begin{proof}[Proof of Proposition \ref{prop:B2critical}]
Fix $\Phi\in \mathscr X$. {Recall that $\mathscr{X}$ has been introduced in Lemma~\ref{Le:Differentiability2D}}. The derivative of the volume $\mathrm{Vol}:\O\mapsto |\O|$ at $\B_2$ in the direction $\Phi$ is given by 
\begin{equation}\label{Eq:DerVolume} \langle d\mathrm{Vol}(\B_2),\Phi\rangle=\int_{\partial \B_2}\langle \Phi,\nu\rangle.\end{equation} For the sake of notational convenience, introduce  
\[ \varphi:=\langle \Phi,\nu\rangle.\] Then, \eqref{Eq:DerVolume} and Proposition \ref{Pr:DifferentiabilityFormulae} yield
\[\langle d\mathcal F(\B_2),\Phi\rangle=\lambda_1(\B_2)\int_{\partial \B_2} \varphi-|\B_2| \int_{\partial \B_2}\Vert (\n \u_{\B_2})\nu\Vert^2\varphi.\]
Consequently, $\B_2$ is a critical point for $\mathcal F$ if, and only if, 
\begin{equation}\label{Eq:Opt} 
|\B_2|\cdot\Vert (\n \u_{\B_2})\nu\Vert^2=\lambda_1(\B_2).
\end{equation} 

{ Let us now prove that \eqref{Eq:Opt} indeed holds. Recall from Remark~\ref{rk:usefulComp} that if we introduce $f(r):=J_1(j_{1,1}r)$ then we have, on $\partial \O$, 
\[ (\n \u_{\B_2})\nu= cf'(1)\begin{pmatrix}-\sin(\theta)\\ \cos(\theta)\end{pmatrix},\text{ with }c=\frac1{\sqrt{\pi}|J_0(j_{1,1})|}\text{ whence }\Vert (\n \u_{\B_2})\nu\Vert^2=\frac{f'(1)^2}{\pi J_0(j_{1,1})^2}.\]
 Since $J_1'(x)=J_0(x)-J_1(x)/x$ for any $x>0$, we further derive
\[f'(1)=j_{1,1} J_1'(j_{1,1})=j_{1,1} J_0(j_{1,1}).\]
Consequently, 
\[ \Vert (\n \u_{\B_2})\nu\Vert^2=\frac{f'(1)^2}{\pi J_0(j_{1,1})^2}=\frac{j_{1,1}^2}{\pi}=\frac{j_{1,1}^2}{|\B_2|}.\]
The conclusion follows.
}
\end{proof}

\subsection{The ball satisfies strong second-order conditions for Problem \eqref{Eq:Pv}}
In this section, we turn to second-order optimality conditions for $\mathcal F$ at $\B_2$. 

\paragraph{Expression of the second order derivative of $\mathcal F$ at $\B_2$.}
By the Hadamard structure theorem, since $\B_2$ is a critical point of $\mathcal F$, the second order derivative $\langle d^2\mathcal F(\B_2)\Phi,\Phi\rangle$ only depends on the normal trace $\varphi=\langle \Phi,\nu\rangle$ of $\Phi$. Let us show the following:
\begin{lemma}\label{Le:Der2}
For any $\Phi\in \mathscr X$, there holds
\[ \langle d^2\mathcal F(\B_2)\Phi,\Phi\rangle=2\lambda_1(\B_2)\left(\int_{\partial \B_2}\varphi^2\right)-\frac{2\lambda_1(\B_2)}{\pi}\left(\int_{\partial \B_2}\varphi\right)^2+2\pi I[\u_\Phi']\] where $I$ is a quadratic form defined as
\[I[\u_\Phi']=\int_{\B_2}\Vert \n \u_\Phi'\Vert^2-\lambda_1(\B_2)\int_{\B_2}\Vert \u_\Phi'\Vert^2.\]
\end{lemma}

\begin{proof}[Proof of Lemma \ref{Le:Der2}]
Observe that we have 
\[ \langle d^2\mathcal F(\B_2)\Phi,\Phi\rangle=\langle d^2\mathrm{Vol}(\B_2)\Phi,\Phi\rangle\lambda_1(\B_2)+2\langle d\mathrm{Vol}(\B_2),\Phi\rangle\langle d\lambda_1(\B_1),\Phi\rangle+|\B_2|\cdot   \langle d^2\lambda_1(\B_2)\Phi,\Phi\rangle .\]
However, we have 
\[ \langle d^2\mathrm{Vol}(\B_2)\Phi,\Phi\rangle=\int_{\partial \B_2} H\varphi^2.\]Using Proposition \ref{Pr:DifferentiabilityFormulae}, { the fact that $p_{\B_2}$, the pressure, is constant (see Lemma~\ref{Le:Multiplicity2d})} and that $H=1$ we obtain
\begin{align*}
\langle d^2\mathcal F(\B_2)\Phi,\Phi\rangle&=\lambda_1(\B_2)\int_{\partial \B_2}\varphi^2-2\int_{\partial \B_2}\varphi\int_{\partial \B_2}\Vert (\n \u_{\B_2})\nu\Vert^2\varphi
\\&+2|\B_2|\cdot\int_{\partial \B_2}\langle (\n \u_{\Phi}')\nu,\u_\Phi'\rangle+|\B_2|\cdot\int_{\partial \B_2}\Vert (\n \u_{\B_2})\nu\Vert^2\varphi^2.
\end{align*}
From \eqref{Eq:Opt} this equation simplifies as
\begin{align*}
\langle d^2\mathcal F(\B_2)\Phi,\Phi\rangle&=2\lambda_1(\B_2)\left(\int_{\partial \B_2}\varphi^2\right)-\frac{2\lambda_1(\B_2)}{|\B_2|}\left(\int_{\partial \B_2}\varphi\right)^2+2|\B_2|\cdot\int_{\partial \B_2}\langle (\n \u_{\Phi}')\nu,\u_\Phi'\rangle
\\&=2\lambda_1(\B_2)\left(\int_{\partial \B_2}\varphi^2\right)-\frac{2\lambda_1(\B_2)}{\pi}\left(\int_{\partial \B_2}\varphi\right)^2+2\pi\cdot\int_{\partial \B_2}\langle (\n \u_{\Phi}')\nu,\u_\Phi'\rangle.
\end{align*}
Now, multiply \eqref{Pb:der} by ${\u_\Phi'}$ and integrate by parts. We obtain
\begin{align*}
\int_{\B_2} \langle{\u_\Phi'}, \Delta {\u_\Phi'}\rangle&= \int_{\B_2} \langle{\u_\Phi'} ,\nabla p'\rangle-\lambda_1(\B_2)\int_{\B_2} \Vert{\u_\Phi'}\Vert^2-\langle d\lambda_1(\B_1),\Phi\rangle\int_{\B_2} \langle{\u}, {\u_\Phi'}\rangle\\
&=-\lambda_1(\B_2)\int_{\B_2} \Vert{\u_\Phi'}\Vert^2.
\end{align*}
To go from the first to the second line, we used the orthogonality of $\u_{\B_2}$ and of $\u_\Phi'$, as well as the following fact:
\begin{align*}
\int_{\B_2} \langle {\u_\Phi'}, \nabla p'\rangle&= -\int_{\B_2} p'\div{\u_\Phi'}+\int_{\partial{\B_2}}p'\langle{\u_\Phi'}, {\nu}\rangle
\\&=\int_{\partial{\B_2}}p'\langle{\u_\Phi'},\nu\rangle\\
&= -\int_{\partial{\B_2}}p'\langle(\nabla {\u_{\B_2}}) \nu, {\nu} \rangle\varphi
\\&=0.
\end{align*}
We used again, in this last part, the fact that $\langle(\n \u_{\B_2})\nu,\nu=0\rangle$ (see Remark \ref{Re:Tangential}). Consequently, we obtain 
\[\int_{\partial \B_2}\langle (\n \u_{\Phi}')\nu,\u_\Phi'\rangle=\int_{\B_2}\Vert \n \u_\Phi'\Vert^2+\int_{\B_2} \langle{\u_\Phi'}, \Delta {\u_\Phi'}\rangle=\int_{\B_2}\Vert \n \u_\Phi'\Vert^2-\lambda_1(\B_2)\int_{\B_2}\Vert \u_\Phi\Vert^2.\]

\end{proof}
Our main lemma is the following:
\begin{lemma}\label{Le:Coercivity}
There exists $c_0>0$ such that, for any $\Phi\in \mathscr X$ such that $\varphi=\langle \Phi,\nu\rangle\in\{1,\cos,\sin\}^\perp$ (for the $L^2$ scalar product on $L^2(\partial \B_2)$) there holds
\[ \langle d^2\mathcal F(\B_2)\Phi,\Phi\rangle\geq c_0\Vert \varphi\Vert_{W^{\frac12,2}(\partial \B_2)}^2.\]
\end{lemma}

The proof of this lemma relies on a diagonalisation of the quadratic form $I$.
\begin{remark}
The vector fields $\Phi$ such that $\varphi \in \operatorname{Span}\{\cos,\sin,1\}$ correspond to translation and dilations; of course, $\mathcal F$ is constant along such deformations, whence the need to assume $\varphi\in \operatorname{Span}\{1,\sin,\cos\}^\perp$ to obtain coercivity.
\end{remark}

\paragraph{Diagonalisation of the quadratic form $I$.}
 Write ${\u}_{\B_2}=(u_1,u_2)^\top$, and $\u_\Phi'\color{black}=(u_{\Phi,1}',u_{\Phi,2}')^\top$. With a slight abuse of notation, we write, in polar coordinates, ${\u}_{\B_2}(x)={\u}(r,\theta)=\phi_{0,1}(r,\theta)\et$, where $\phi_{0,1}$ is given by \eqref{Stokes_eig0} { and $\et=\begin{pmatrix}-\sin \theta\\ \cos\theta\end{pmatrix}$ }. Similarly, as we know that we can assume the perturbation $\Phi$ to be normal, we identify $\Phi$ with $\varphi$, and we decompose $\varphi:\partial \B_2\to \R$ as a Fourier series by considering the two sequences  $(\alpha_n)_{n\in \N}$ and $(\beta_n)_{n\in\N}$ in $\ell^2(\N)$ such that
\begin{equation}\label{exp:varphi}
\varphi(\theta)=\sum_{n=0}^{+\infty}(\alpha_n\cos (n\theta)+\beta_n\sin (n\theta)).
\end{equation}
Note that, since $J_1=-J_0'$, one has in particular 
\begin{equation}\label{expr:u1u2}
u_1(r,\theta)=c_1J_1(j_{1,1}r)\sin \theta\qquad \text{and}\qquad u_2(r,\theta)=c_1J_1(j_{1,1}r)\cos \theta
\end{equation}
where 
\begin{equation}\label{def:c1}
{c_1=\frac{1}{\sqrt{\pi}|J_0(j_{1,1})|}.}
\end{equation}
As $\div(\u_\Phi')=0$, we can write
\[ \u_\Phi'=\begin{pmatrix}-\partial_y \psi_\Phi'\\ \partial_x \psi_\Phi'\end{pmatrix}\] where $\psi_{\Phi'}\in W^{2,3}(\B_2)$.  Taking the $\curl$ of \eqref{Pb:der} it appears that $\psi_\Phi'$ satisfies

\[\Delta^2\psi_\Phi'+\lambda_1(\B_2)\Delta \psi_\Phi'=0\quad \text{in }\B_2\color{black}.
\]
As $\Delta \psi_\Phi'+\lambda_1(\B_2)\psi_\Phi'$ is harmonic, there exist $(c_n)_{n\in \N}\,, (d_n)_{n\in \N}\in \ell^2(\N)$ such that\[
\Delta\psi_\Phi'+\lambda_1(\B_2)\psi_\Phi'={\lambda_1(\B_2)}\sum_{n=0}^{+\infty} ( c_n\cos (n\theta)+ d_n\sin (n\theta))r^n.
\]

{Observe that the function $G$ given in polar coordinates by $G:(r,\theta)\mapsto \sum_{n=0}^{+\infty} ( c_n\cos (n\theta)+ d_n\sin (n\theta))r^n$ is a particular solution of this equation in $\B_2$. As the function $z$ given in polar coordinates by $z:(r,\theta)\mapsto \psi_\Phi'(r,\theta)-G(r,\theta)$ solves $\Delta z+\lambda_1(\B_2)z=0$, a separation of variables  yields
the existence of two sequences $(a_n)_{n\in \N}\,, (b_n)_{n\in \N}\in \ell^2(\N)$ such that}
\begin{equation}\label{expr:psi0844}
\psi_\Phi'(r,\theta)=\sum_{n=0}^{+\infty}\left((a_n \cos (n\theta)+b_n\sin (n\theta))J_n(j_{1,1}r)+(c_n \cos (n\theta)+d_n\sin (n\theta))r^n\right).
\end{equation}
Let us now identify all coefficients.

{\begin{lemma}\label{lem:coefftsanbncndc}
For every $n\geq 2$, one has
\begin{eqnarray}
a_n&=&\frac{j_{1,1}}{\sqrt{\pi}\left(nJ_n(j_{1,1})-j_{1,1}J_n'(j_{1,1})\right)}\alpha_n,\\
b_n&=&\frac{j_{1,1}}{\sqrt{\pi}\left(nJ_n(j_{1,1})-j_{1,1}J_n'(j_{1,1})\right)}\beta_n,\\
c_n&=&-\frac{j_{1,1}J_n(j_{1,1})}{\sqrt{\pi}\left(nJ_n(j_{1,1})-j_{1,1}J_n'(j_{1,1})\right)}\alpha_n,\\
d_n&=&-\frac{j_{1,1}J_n(j_{1,1})}{\sqrt{\pi}\left(nJ_n(j_{1,1})-j_{1,1}J_n'(j_{1,1})\right)}\beta_n.
\end{eqnarray}
\end{lemma}}

{\begin{proof}[Proof of Lemma~\ref{lem:coefftsanbncndc}]
In what follows, to determine the coefficients $a_n$, $b_n$, $c_n$, $d_n$, we will exploit \eqref{expr:psi0844} and the boundary conditions
$$
u_1'=-\frac{\partial\psi_{\Phi}'}{\partial y}=-\frac{\partial u_1}{\partial\nu}\varphi \quad \text{on }\partial\B_2\qquad \text{and}\qquad u_2'=\frac{\partial\psi_{\Phi}'}{\partial x}=-\frac{\partial u_2}{\partial\nu}\varphi \quad \text{on }\partial\B_2.
$$
As $J_1'(x)=J_0(x)-J_1(x)/x$ for any $x>0$, we deduce that $J_1'(j_{1,1})=J_0(j_{1,1})<0$. From Remark~\ref{rk:usefulComp}, one gets
$$
\begin{pmatrix}
u_1'\\ u_2'
\end{pmatrix}=\frac{J_0(j_{1,1})}{|J_0(j_{1,1})|}\cdot \frac{j_{1,1}}{\sqrt{\pi}}\begin{pmatrix}
-\sin \theta\\ \cos \theta
\end{pmatrix} \varphi(\theta) =
-\frac{j_{1,1}}{\sqrt{\pi}}\begin{pmatrix}
-\sin \theta\\ \cos \theta
\end{pmatrix} \varphi(\theta)\quad\text{on }\partial\B_2.
$$
Using that $u_1'=-\partial \psi_{\Phi}'/\partial y$ and $u_2'=\partial \psi_{\Phi}'/\partial x$, it follows that
\begin{eqnarray*}
u_1'&=&\left.-\sin\theta \frac{\partial \psi_{\Phi}'}{\partial r}-\cos\theta \frac{\partial \psi_{\Phi}'}{\partial \theta}\right|_{r=1} = \sin\theta \,
\frac{j_{1,1}}{\sqrt{\pi}}\, \varphi(\theta) \\
u_2'&=&\left.\cos\theta \frac{\partial \psi_{\Phi}'}{\partial r}-\sin\theta \frac{\partial \psi_{\Phi}'}{\partial \theta}\right|_{r=1} =- \cos\theta \,
\frac{j_{1,1}}{\sqrt{\pi}}\, \varphi(\theta) \\
\end{eqnarray*}
From these two equations, we immediately obtain
$$
\frac{\partial \psi_{\Phi}'}{\partial r} = - \frac{j_{1,1}}{\sqrt{\pi}}\, \varphi(\theta) \qquad \text{and}\qquad \frac{\partial \psi_{\Phi}'}{\partial \theta}  = 0.
$$
Now since
$$ \frac{\partial \psi_{\Phi}'}{\partial r} = \sum_{n=0}^{+\infty}  j_{1,1}J_n'(j_{1,1})(a_n\cos (n\theta)+b_n\sin (n\theta))+n (c_n\cos (n\theta)+d_n\sin (n\theta))$$
and
$$\frac{\partial \psi_{\Phi}'}{\partial \theta}= \sum_{n=0}^{+\infty} n J_n(j_{1,1})(b_n\cos (n\theta)-a_n\sin (n\theta))+n
(d_n\cos (n\theta)-c_n\sin (n\theta))$$
we obtain, by identification the four relations
\begin{eqnarray}
j_{1,1} J_n'(j_{1,1}) a_n +n c_n= - \frac{j_{1,1}}{\sqrt{\pi}} \alpha_n \\
j_{1,1} J_n'(j_{1,1}) b_n +n d_n= - \frac{j_{1,1}}{\sqrt{\pi}} \beta_n \\
J_n(j_{1,1}) b_n + d_n =0\\
J_n(j_{1,1}) a_n + c_n =0
\end{eqnarray}
We obtain the desired expression by solving the two systems in $a_n,c_n$ and $b_n,d_n$ respectively.
\end{proof}}

\begin{proof}[Proof of Lemma \ref{Le:Coercivity}]
Let us compute $I$: we have
\begin{eqnarray*}
I[\u_\Phi'] &=& \int_{\partial\B_2}\left(u_1'\frac{\partial u_1'}{\partial \nu}+u_2'\frac{\partial u_2'}{\partial \nu}\right)\, d\sigma\\
&=& -\frac{j_{1,1}}{\sqrt{\pi}}\int_{\partial\B_2}\varphi \sin \theta\left[\sin\theta\frac{\partial^2\psi_\Phi'}{\partial r^2}-\cos \theta\frac{\partial\psi_\Phi'}{\partial \theta}+\cos \theta\frac{\partial ^2\psi_\Phi'}{\partial r\partial \theta} \right]\, d\sigma\\
&& +\frac{j_{1,1}}{\sqrt{\pi}}\int_{\partial\B_2}\varphi\cos\theta \left[-\cos\theta\frac{\partial^2\psi_\Phi'}{\partial r^2}-\sin \theta\frac{\partial\psi_\Phi'}{\partial \theta}+\sin \theta\frac{\partial ^2\psi_\Phi'}{\partial r\partial \theta} \right]\, d\sigma\\
&=& \frac{j_{1,1}}{\sqrt{\pi}}\int_{\partial\B_2}\varphi \frac{\partial^2\psi_\Phi'}{\partial r^2}.
\end{eqnarray*}
Using the expansion \eqref{expr:psi0844} of $\psi$, we get
$$
\left.\frac{\partial^2\psi_\Phi'}{\partial r^2}\right|_{\partial\B_2}=\sum_{n=0}^{+\infty}j_{1,1}^2(a_n\cos \theta+b_n\sin\theta)J_n''(j_{1,1})+\sum_{n=1}^{+\infty}n(n-1)(c_n\cos (n\theta)+d_n\sin(n\theta)),
$$
and we thus obtain
\begin{eqnarray*}
I[\u_\Phi'] &=& -\frac{j_{1,1}}{\sqrt{\pi}}\left[\pi j_{1,1}^2\sum_{n=1}^{+\infty}(a_n\alpha_n+b_n\beta_n)J_n''(j_{1,1})+\pi \sum_{n=1}^{+\infty}n(n-1)(c_n\alpha_n+d_n\beta_n)\right]\\
&=& -j_{1,1}^2\sum_{n=1}^{+\infty}\frac{j_{1,1}^2J_n''(j_{1,1})-n(n-1)J_n(j_{1,1})}{nJ_n(j_{1,1})-j_{1,1}J_n'(j_{1,1})}(\alpha_n^2+\beta_n^2).
\end{eqnarray*}

Now, by definition of Bessel functions of the first order, one has 
$$
J_n''(x)+\frac{1}{x}J_n'(x)+\left(1-\frac{n^2}{x^2}\right)J_n(x)=0
$$ 
for every $x>0$. We infer that $j_{1,1}^2J_n''(j_{1,1})-n(n-1)J_n(j_{1,1})=(n-j_{1,1}^2)J_n(j_{1,1})-j_{1,1}J_n'(j_{1,1})$, yielding finally
\begin{equation}\label{expr:I:0924}
I=-j_{1,1}^2\sum_{n=1}^{+\infty}\frac{(n-j_{1,1}^2)J_n(j_{1,1})-j_{1,1}J_n'(j_{1,1})}{nJ_n(j_{1,1})-j_{1,1}J_n'(j_{1,1})}(\alpha_n^2+\beta_n^2).
\end{equation}
Recall that we assume $\varphi\in\langle 1\rangle^\perp$ where $1$ is the constant function and  $\langle 1\rangle^\perp$ is its $L^2$ orthogonal subspace.

According to \eqref{expr:I:0924}, one has
\begin{eqnarray*}
\langle d^2\mathcal F(\B_2)\Phi,\Phi\rangle &=& 2\pi I[\u_\Phi']+2\lambda_1(\B_2) \int_{\partial\B_2}\varphi^2-\frac{2\lambda_1(\B_2)}{\pi}\left(\int_{\partial\B_2}\varphi\right)^2\\
&=& 2\pi j_{1,1}^2\sum_{n=1}^{+\infty}\frac{nJ_n(j_{1,1})-j_{1,1}J_n'(j_{1,1})-(n-j_{1,1}^2)J_n(j_{1,1})-j_{1,1}J_n'(j_{1,1})}{nJ_n(j_{1,1})+j_{1,1}J_n'(j_{1,1})}(\alpha_n^2+\beta_n^2).
\end{eqnarray*}
Furthermore, recall that $xJ_n'(x)=nJ_n(x)-xJ_{n+1}(x)$ for every $x\geq 0$, so that $j_{1,1}J_n'(j_{1,1})-nJ_n(j_{1,1})=-j_{1,1}J_{n+1}(j_{1,1})$.

This leads to
\begin{equation}
\langle d^2\mathcal F(\B_2)\Phi,\Phi\rangle =2\pi j_{1,1}^4\sum_{n=2}^{+\infty} \gamma_n(\alpha_n^2+\beta_n^2)\qquad \text{with }\gamma_n=\frac{J_n(j_{1,1})}{j_{1,1}J_{n+1}(j_{1,1})}.
\end{equation}
(we recall that we have supposed $\varphi$ orthogonal to $1,\cos\theta,\sin\theta$, so $a_0=a_1=b_1=0$.)
Since the first positive zero of $J_n$ is greater than $j_{1,1}$ for all $n\geq 2$, one has $\gamma_n>0$ for every $n\geq 2$. 
Furthermore, using that $J_n(x)\sim \frac{1}{n!}\left(\frac{x}{2}\right)^n$ as $n\to +\infty$, we infer that 
$$
\gamma_n \sim 2(n+1)/j_{1,1}^2\quad { as }\quad n\to +\infty.
$$
Since the $H^{1/2}$ norm is given by
$$\|\varphi\|_{H^{1/2}}=\sum_{n=2}^\infty n(a_n^2+b_n^2)$$
the conclusion follows.
\end{proof}

\section{Necessary optimality conditions: proof of Theorem \ref{Th:Necessary}}\label{secproof:necessaryOptCond}
As alluded to in the introduction, Theorem \ref{Th:NonOptimality3d} follows from Theorem \ref{Th:Necessary}, so that it is clearer to begin with the proof of the latter (nevertheless, we give a proof of {Theorem~\ref{Th:NonOptimality3d}} in
section \ref{sec:proofTheo3D} that only relies on the computations of the semi-differential analysed in the present section). 
This theorem relies on several results on the semi-differentiability of eigenvalues. To write things down in a precise way, we consider a domain $\O$ whose boundary is of class $\mathscr C^{2,\alpha}$ with $\alpha\in (0,1)$. The solution $\u$ thus belongs to $C^{2,\alpha}(\overline{\O})$ by standard elliptic regularity.  
Although $\lambda_1(\O)$ may not be simple, it has finite multiplicity, say $N$, and the map $\mathcal F$ is semi-shape differentiable at $\O$  \cite[Theorem 2.1, Chapter 1]{DelfourZolesio} (see also \cite[Theorem~2.1 and Remark~2.2]{Caubet_2021} and \cite{zbMATH01167545}), in the sense that, for any vector field $\Phi\in \mathscr X$, the limit
\[\langle \partial\mathcal F(\O),\Phi\rangle:=\underset{t\searrow 0}\lim\frac{\mathcal F((\mathrm{Id}+t\Phi)\O)-\mathcal F(\O)}t\] exists. This fact relies on the semi-differentiability of $\lambda_1$.  By  \cite[Theorem 2.1, Chapter 1]{DelfourZolesio}, if we let $E_1$ denote the first eigenspace associated with $\lambda_1(\O)$, there holds 
\[ \langle \partial \lambda_1(\O),\Phi\rangle=-\min_{\u \in E_1\,, \int_\O \Vert \u\Vert^2=1}\int_{\partial \O} \Vert (\n\u)\nu\Vert^2\langle \Phi,\nu\rangle.\]
Let $\{\u_1,\dots,\u_N\}$ be an orthonormal basis of $E_1$. Using the differentiability of the volume, we deduce that 
\begin{align*}
\langle \partial \mathcal F(\O),\Phi\rangle&=\frac23\lambda_1(\O)|\O|^{-\frac13}\int_{\partial \O} \langle \Phi,\nu\rangle-|\O|^{\frac23}\min_{\u \in E_1\,, \int_\O \Vert \u\Vert^2=1}\int_{\partial \O} \Vert (\n\u)\nu\Vert^2\langle \Phi,\nu\rangle
\\&=\min_{\u \in E_1\,, \int_\O \Vert \u\Vert^2=1}\int_{\partial \O} \left(\frac23\lambda_1(\O)|\O|^{-\frac13}-|\O|^{\frac23}\cdot\Vert (\n\u)\nu\Vert^2\right)\langle \Phi,\nu\rangle
\\&=\min_{(\alpha_i)_{i=1,\dots,N}, \ \sum_{i=1}^{N}{\alpha_i^2}=1}\sum_{i,j=1}^{N}\int_{\partial \O} \left(\frac23\lambda_1(\O)|\O|^{-\frac13}\alpha_i^2-|\O|^{\frac23} \alpha_i\alpha_j\langle (\n\u_i)\nu,(\n\u_j)\nu\rangle\right)\langle \Phi,\nu\rangle.
\end{align*}
Let 
\[ M_\Phi:=\frac23\lambda_1(\O)|\O|^{-\frac13}\int_{\partial \O}\langle \Phi,\nu\rangle\mathrm{I}_3-|\O|^{\frac23} \left(\int_{\partial \O}\langle (\n\u_i)\nu,(\n\u_j)\nu\rangle\langle\Phi,\nu\rangle
\right)_{1\leq i,j\leq N}.\] $I_3$ is the $3\times 3$ identity matrix. $M_\Phi$ is a real, symmetric matrix, and is thus diagonalisable. Let $\Sigma_\Phi$ be its spectrum. If $\O^*$ is optimal, then the optimality conditions read:
\[ \forall \Phi\in \mathscr X\,, \Sigma_\Phi\subset [0;+\infty).\] However, as $\Sigma_{-\Phi}=-\Sigma_\Phi$, we deduce that, if $\O$ is optimal then
\[\forall \Phi\in \mathscr X\,, \Sigma_\Phi=\{0\}.\]
As $M_\Phi$ is symmetric, we deduce that, for any $\Phi$, there holds
\[ M_\Phi=0.\] 
We thus obtain the following fact: if $\O^*$ is optimal then for any $(i,j)\in \{1,\dots,N\}^2$, for any $x\in \partial \O^*\color{black}$,
\[ \langle (\n \u_i)\nu,(\n \u_j)\nu\rangle=\frac23\cdot\frac{\lambda_1(\O^*)}{|\O^*|}\delta_{i,j}=\frac1{\sigma_{\O^*}^2}\delta_{i,j}.\] In this expression, $\delta_{i,j}$ is the usual Kronecker symbol and ${\sigma_{\O^*}^2}=\frac32\cdot\frac{|\O^*|}{\lambda_1(\O^*)}$. 

However, since each $\u_i$ is divergence free we have $\langle (\n \u_i)\nu,\nu\rangle=0$ (see Remark \ref{Re:Tangential}). In particular, the multiplicity of the first eigenspace $E_{1}$ is at most 2 (we recall that we work in dimension 3).

Now, let us argue by contradiction and assume that $\lambda_1(\O^*)$ has multiplicity 2. Let $(\u_1,\u_2)$ be an orthonormal basis of $E_1$. From the previous computation, we deduce that, for any $x\in \partial \O^*$, $\{\sigma_{\O^*}(\n \u_1)\nu,\sigma_{\O^*}(\n \u_2)\nu\}$ is an orthonormal basis of the tangent plane $T_x(\partial \O)$. Let us now argue using a scheme of proof inspired by \cite[Proof of Theorem 2]{Gerner_2023} to reach a contradiction. 

For $i=1,2$, let $\v_i$ be the solution of
\begin{equation}\label{Eq:vi}
\begin{cases}-\Delta \v_i+\n p_i=\lambda_1(\O^*)\v_i+\omega \u_i&\text{ in }{\O^*}\,, 
\\ \n\cdot \v_i=0&\text{ in }{\O^*}\,, 
\\ \v_i\in E_{1}^\perp\,, 
\\ \v_i=-(\n \u_i)\nu&\text{ on }\partial {\O^*},\end{cases}\end{equation}
where  $E_1^\perp$ is the $L^2$-orthogonal subspace to $E_1$ and {the constant $\omega$ is defined by}
\begin{equation}\label{Eq:Omega} 
\omega=-\int_{\partial \O^*}\Vert (\n \u_i)\nu\Vert^2=-\frac23\mathrm{Per}(\partial \O^*)\cdot\frac{\lambda_1(\O^*)}{|\O^*|}.\end{equation}
We have the following result:
\begin{lemma}\label{Le:ViWd}
For $i=1,2$, $\v_i$ is uniquely defined.\end{lemma}
\begin{proof}[Proof of Lemma \ref{Le:ViWd}]
The uniqueness of $\v_i$ follows from the Fredholm alternative. Regarding the existence of $\v_i$, it is clear that the energy functional 
\[ \mathcal E:W^{1,2}_0(\O;\R^3)\color{black}\ni \v\mapsto \frac12\int_{\O^*} \Vert \n \v\Vert^2-\frac{\lambda_1(\O^*)}{2}\int_{\O^*}\Vert \v\Vert^2\] is coercive on 
\[ X_i:=\{\v \in W^{1,2}_0(\O;\R^3)\color{black}\,, \n\cdot \v=0\,, \v=-(\n \u_i)\nu\text{ on }\partial {\O^*}\}\cap E_{1}^\perp.\]
We let $\v_i$ be a minimiser of $\mathcal E$ on $X_i$. In particular, from the Euler-Lagrange equation we deduce that 
\[ -\Delta \v_i-\lambda_1(\O^*)\v_i\in \left(\{\v \in W^{1,2}_0(\O;\R^3)\color{black}\,, \n\cdot \v=0\}\cap E_{1}^\perp\right).\]
Thus, there exist a function $p_i$ and two coefficients $(\alpha_{i,j})_{j=1,2}$ such that 
\[ -\Delta \v_i-\lambda_1(\O^*)\v_i=-\n p_i+\sum_{j=1,2}\alpha_{i,j}\u_j.\] Multiplying by $\u_j$, using the fact that $(\n \u_i)\nu$ and $(\n \u_j)\nu$ are orthogonal on $\partial \O^*$, we deduce that 
\[ \alpha_{i,j}=\omega \delta_{i,j}\] where $\omega$ is defined by \eqref{Eq:Omega}.
\end{proof}

{Turning back to the proof of Theorem \ref{Th:Necessary}, let us write $\v_i=(v_{i,k})_{1\leq k\leq 3}$ for $i=1,2$ and consider the scalar function $\langle \v_1,\v_2\rangle$ defined by $\langle \v_1,\v_2\rangle(\cdot)=\sum_{k=1}^3v_{1,k}(\cdot)v_{2,k}(\cdot)$.}

{Then, denoting by $\n (\langle \v_1,\v_2\rangle)$ the gradient of $ \langle \v_1,\v_2\rangle$ and still retaining, with a slight abuse of notations, the notation $\n \v_i$ for the Jacobian matrix of $\v_i$, observe that }
\[ \n (\langle \v_1,\v_2\rangle)=(\n \v_1)\v_2+(\n \v_2)\v_1.\] Since $\v_1$ and $\v_2$ are orthogonal on $\partial \O^*$, we deduce that, letting $X^\tau$ denote the tangential part of a vector field $X$ on $\partial \O^*$, there holds
\[\left((\n \v_1)\v_2\right)^\tau=-\left((\n \v_2)\v_1\right)^\tau.\]

Furthermore, since $\n \langle \v_1,\v_2\rangle$ is colinear to $\nu$, and since $\v_1$ is orthogonal to $\nu$ there holds
\[ 0=\langle \v_1, \n\langle \v_1,\v_2\rangle\rangle.\]
Developing the gradient yields
\begin{align*}
0&=\langle \v_1, \n\langle \v_1,\v_2\rangle\rangle.=\langle \v_1, (\n \v_1)\v_2\rangle+\langle \v_1,(\n \v_2)\v_1\rangle
\\&=\langle \v_1,(\n \v_2)\v_1\rangle+\sum_{i,j} \v_{1,i}\v_{2,j}\frac{\partial \v_{1,i}}{\partial x_j}=\langle \v_1,(\n \v_2)\v_1\rangle+\frac12\sum_{i,j} \v_{2,j}\frac{\partial \v_{1,i}^2}{\partial x_j}
\\&=\langle \v_1,(\n \v_2)\v_1\rangle
\end{align*}
a.e. on $\partial\O$, where we used once again the optimality condition to write that $\Vert \v_1\Vert^2$ is constant on $\partial \O^*$, and the fact that $\v_2$ is tangential.

Similarly, 
\[ 0=\langle \v_2,(\n \v_1)\v_2\rangle=-\langle \v_2,(\n \v_2)\v_1\rangle.\] Here again the fact that $\v_2$ is tangential is instrumental. Thus $(\n \v_2)\v_1$ is orthogonal to both $\v_1$ and $\v_2$, so that there exists a function $f:\partial\O\to \R$ satisfying
\[
{(\n \v_2(x))\v_1(x)=f(x) \nu(x), \qquad x\in \partial\O.}
\] 
For the same reason, 
\[
{ (\n \v_1(x))\v_2(x)=g(x)\nu(x), \qquad \text{a.e. }x\in \partial\O}
\] for some function $g$. Finally, letting $[X,Y]$ be the Lie bracket of two vector fields,
\[ 
{[\v_1,\v_2]=(\n \v_1)\v_2-(\n \v_2)\v_1=(f-g)\nu\qquad \text{on }\partial\O.}
\] 
However, as $\v_1$ and $\v_2$ are both tangential, so is their Lie bracket, whence $[\v_1,\v_2]=0$. In particular, this implies that $\partial \O^*$ admits a set of two pointwise orthogonal, commuting vector fields. This implies that the Riemann tensor associated with  $\partial \O^*$  endowed with the induced metric is zero (we refer for instance to \cite[ Theorem 3.1.7]{Kunzinger} or \cite[Theorem 7.3 and its proof]{zbMATH01077335}). In particular, the Gau\ss\, curvature of $\partial \O^*$ is zero everywhere, which is impossible.
%
%
%

\section{The ball does not minimise $F$ in dimension 3: proof of Theorem~\ref{Th:NonOptimality3d}}\label{sec:proofTheo3D}
As in the two-dimensional case, we need to describe precisely the eigenspace associated with $\lambda_1(\B_3)$. The description of this first eigenvalue is less standard and mostly due to Saks \cite{Saks_2013}. However, \cite{Saks_2013} is quite an involved article and, for the sake of convenience, we recall the main steps of \cite{Saks_2013} that we use to derive the results of this preliminary section in Appendix \ref{Ap:Saks}.

Let $\omega>0$ be the first positive root of the equation $\tan(x)=x$. 
The main result is the following:
\begin{proposition}\label{Pr:Saks}
$\lambda_1(\B_3)$ {is equal to $\omega^2$ and} has multiplicity 3. An orthogonal basis of the associated eigenspace $E_1$ is given by 
\[
\left\{\U_1=\psi(r)\begin{pmatrix}0\\z\\-y\end{pmatrix}\,,\U_2=\psi(r)\begin{pmatrix}-z\\0\\x\end{pmatrix}\,, \U_3=\psi(r)\begin{pmatrix}-y\\x\\0\end{pmatrix}\right\}\] where 
\[\psi(r)=\frac{\sin (\omega r)}{\omega^2r^3}-\frac{\cos (\omega r)}{\omega r^2}.
\]
Finally, for any $i\in\{1,2,3\}$, there holds
{\[ \int_{\B_3}\Vert \U_i\Vert^2=\frac{4\pi}{3\omega^4}\left(\omega^2-\sin^2\omega\right)=\frac{4\pi}{3}\cos^2\omega=:A^2.\]}
\end{proposition}
As mentioned, we prove this multiplicity result in Appendix \ref{Ap:Saks}.
 In particular, applying Theorem \ref{Th:Necessary}, the conclusion follows.

\medskip

Let us now give another proof of Theorem~\ref{Th:NonOptimality3d} that only relies on the analysis of the semi-differential as given in section~\ref{secproof:necessaryOptCond}. Recall that the semi-differential of $\lambda_1$ at $\B_3\color{black}$ in direction $\Phi \in  \mathscr X$ (see Lemma~\ref{Le:Differentiability2D} for the definition of $\mathscr X$) reads
\[\langle \partial\lambda_1(\B_3), \Phi\rangle =\min_{\substack{\mathbf{U}\in \operatorname{Span}\{\mathbf{U}_i/A\}_{1\leq i\leq 3}\\ \Vert \mathbf{U}\Vert_{L^2(\B_3)}=1}}-\int_{\partial \B_3}\Vert\nabla \mathbf{U}\nu\Vert^2\langle \Phi,\nu\rangle .
\]
It follows that
\begin{eqnarray*}
\langle \partial\lambda_1(\B_3), \Phi\rangle &=& \min_{\substack{(\alpha_i)_{1\leq i\leq 3}\in \R^3\\ \alpha_1^2+\alpha_2^2+\alpha_3^2=1}}-\frac{1}{A^2}\int_{\partial \B_3}\left\Vert\sum_{i=1}^3\alpha_i \nabla \mathbf{U}_i\nu\right\Vert^2\langle \Phi,\nu\rangle\\
&=& \frac{1}{A^2}\mu_1\left(\mathcal{M}_{\partial \B_3}(\langle \Phi,\nu\rangle)\right),
\end{eqnarray*}
where $\mu_1\left(\mathcal{M}_{\partial \B_3}(\langle\Phi, \nu\rangle)\right)$ denotes the lowest eigenvalue of the matrix
\[
\mathcal{M}_{\partial \B_3}(\langle\Phi, \nu\rangle)=\left(-\int_{\partial \B_3}\left(\nabla \mathbf{U}_i\nu\cdot \nabla \mathbf{U}_j\nu\right) \langle\Phi, \nu\rangle\right)_{1\leq i,j\leq 3}.
\] 
From the explicit expressions for the vector fields $\mathbf{U}_i$ provided in Proposition~\ref{Pr:Saks}, one computes successively 
{\[
\nabla \mathbf{U}_1\nu=\begin{pmatrix}
0 \\ z\psi'(1)\\ -y\psi'(1)
\end{pmatrix}, \quad 
\nabla \mathbf{U}_2\nu=\begin{pmatrix}
-z\psi'(1)\\ 0\\ x\psi'(1)
\end{pmatrix},\quad 
\nabla \mathbf{U}_3\nu=\begin{pmatrix}
-y\psi'(1)\\ x\psi'(1)\\ 0
\end{pmatrix}
\]}
and
$$
{\mathcal{M}_{\partial \B_3}(\langle\Phi, \nu\rangle)=\psi'(1)^2\begin{pmatrix}
-\int_{\partial \B_3}(y^2+z^2)\langle\Phi, \nu\rangle  & \int_{\partial \B_3}xy\langle\Phi, \nu\rangle  &- \int_{\partial \B_3}xz\langle\Phi, \nu\rangle \\
\int_{\partial \B_3}xy\langle\Phi, \nu\rangle  & -\int_{\partial \B_3}(x^2+z^2)\langle\Phi, \nu\rangle  & -\int_{\partial \B_3}yz\langle\Phi, \nu\rangle \\
-\int_{\partial \B_3}xz\langle\Phi, \nu\rangle  & -\int_{\partial \B_3}yz\langle\Phi, \nu\rangle  & -\int_{\partial \B_3}(x^2+y^2)\langle\Phi, \nu\rangle \\
\end{pmatrix}.}
$$
{Furthermore, one easily computes $\psi'(1)=\sin \omega \neq 0$.
Observe that $\mathcal{M}_{\partial \B_3}(\langle\Phi, \nu\rangle)$ is real symmetric. We infer from the previous computations that the semi-differential of $\mathcal F$, defined by \eqref{Eq:Pv}, at $\B_3$ in direction $\Phi$ is given by
\begin{eqnarray*}
\langle \partial\mathcal F(\B_3), \Phi\rangle &=& \frac23 \lambda_1(B_3)|\B_3|^{-1/3}\int_{\partial \B_3}\langle\Phi, \nu\rangle + \frac{|\B_3|^{2/3}}{A^2}\mu_1\left(\mathcal{M}_{\partial \B_3}(\langle\Phi, \nu\rangle)\right)
\end{eqnarray*}
Moreover, since
$$
\frac{|\B_3|}{A^2}=\frac{4\pi}{3}\cdot \frac{3\omega^2}{4\pi\sin^2\omega}=\frac{\omega^2}{\sin^2\omega},
$$
it follows that
\begin{eqnarray*}
\langle \partial\mathcal F(\B_3), \Phi\rangle &=& \omega^2 |\B_3|^{-1/3}\left(\frac23 \int_{\partial \B_3}\langle\Phi, \nu\rangle+\frac{1}{\sin ^2\omega }\mu_1\left(\mathcal{M}_{\partial \B_3}(\langle\Phi, \nu\rangle)\right)\right)\\
&=& \omega ^2|\B_3|^{-1/3}\mu_1\left(\widehat{M}_\Phi\right)
\end{eqnarray*}
where $\widehat{M}_\Phi$ denotes the $3\times 3$ matrix
\[
\widehat{M}_\Phi:=\frac23 \int_{\partial \B_3}\langle\Phi, \nu\rangle \operatorname{I}_3+\frac{1}{\sin ^2\omega}\mathcal{M}_{\partial \B_3}(\langle\Phi, \nu\rangle), 
\]
$\operatorname{I}_3$ is the identity matrix of size 3 and $\mu_1\left(\widehat{M}_\Phi\right)$ denotes the lowest eigenvalue of the matrix $\widehat{M}_\Phi$. }

{Observe that the trace of the matrix $\widehat{M}_\Phi$ is 
\[
\operatorname{Tr}\widehat{M}_\Phi=\left(2-2\frac{\sin^2\omega}{\sin^2\omega}\right)\int_{\partial \B_3}\langle\Phi, \nu\rangle=0,
\]
meaning, as expected, that the functional $\mathcal F$ is dilation invariant. }
Using this observation, we claim that the non optimality of $\B_3$ will be proven whenever one shows that $\widehat{M}_\Phi$ is not the null matrix. Indeed, in that case, since $\widehat{M}_\Phi$ is real symmetric, we will get that $\widehat{M}_\Phi$ has two eigenvalues with opposite sign, yielding the existence of a perturbation which strictly decreases the functional $\mathcal F$. It suffices for instance to find $\Phi$ such that
\[
\int_{\partial \B_3}yz\langle\Phi, \nu\rangle\neq 0.
\]
Note that $\nu=x\vec{e}_x+y\vec{e}_y+z\vec{e}_z$ on $\partial\B_3$, in cartesian coordinates. We choose $\Phi$ given by $\Phi(x,y,z)=z\vec{e}_y$ and we are done.

\section{Conclusion and open problems}
There are several questions that we believe are interesting, but would require further work and, most likely, the use of different tools. Let us give some that we believe are the most ambitious.
\paragraph{Characterising the optimal domain $\O^*$.}
{Although we have proved that, in $\R^2$, the ball $\B_2$ is likely to be an optimiser, it is not clear to us how one would approach such a result. As detailed in Remark~\ref{rk:dim1913}, if we can prove that there exists a simply connected and regular minimiser, then it is necessarily a disc. The results in our paper complete the intuition that the disc minimises the Stokes eigenvalue under volume constraints in dimension 2. However, the proof of this result is an open problem as far as we know.} 

Indeed, in the vectorial case, standard rearrangement and symmetrisation tools are bound to fail. One might then try to use an alternative approach to the Faber-Krahn inequality, typically using an overdetermined problem framework, but, to the best of our knowledge, no generalisation of the Serrin theorem can accommodate incompressibility constraints. In three dimensions, it seems useful to perform some numerical simulations to
have at least an idea of the possible optimal sets. Following our Corollary \ref{Co:HairyBall}, one can wonder whether it is a particular torus;

\paragraph{\emph{A priori} regularity of the optimal domain.}
Following Theorem \ref{Th:Existence}, we know that an optimal domain $\O^*$ exists in the class of quasi-open sets. Nevertheless, it seems extremely ambitious, at this stage, to develop an \emph{a priori} regularity theory for this optimal set. Following the usual approach to such problems,
see e.g. \cite[chapter3]{bookHedG}, the first step in that direction would be to establish that $\O^*$ is, in fact, open. To derive further properties, the optimality conditions given in Theorem \ref{Th:Necessary} would probably play a crucial role in the development of an appropriate blow-up theory.

\section*{Acknowledgments}
The authors warmly thank Dorin Bucur, who provided the main argument in the proof of existence. The authors also thank Davide Buoso, Wadim Gerner and Antoine Lemenant for helpful conversations and their comments on the results of this article.  Finally, the authors wish to warmly thank the anonymous reviewer, whose comments helped improve and clarify this article.
\color{black}

 I. M-F was partially funded by a PSL Young Researcher Starting Grant 2023. The support of the CNRS through a 2023 PEPS project is also acknowledged.  

\section*{Data Availability Statement}
No datasets were generated or analysed during the current study.

\appendix

\begin{center}
\Large{\textbf{\fbox{Appendix}}}
\end{center}

\section{The first Stokes eigenvalue in $\B_3$}\label{Ap:Saks}
In this appendix, we collect several facts about $\lambda_1(\B_3)$ and its associated eigenfunctions. The main goal is to prove the following proposition:
\begin{proposition}\label{Pr:StokesBall} Let $\omega>0$ be the first positive root of $tan(x)=x$. Then $\lambda_1(\B_3)$ {is equal to $\omega^2$ and} has multiplicity 3. An orthogonal basis of the associated eigenspace $E_1$ is given by 
\[
\left\{\U_1=\psi(r)\begin{pmatrix}0\\z\\-y\end{pmatrix}\,,\U_2=\psi(r)\begin{pmatrix}-z\\0\\x\end{pmatrix}\,, \U_3=\psi(r)\begin{pmatrix}-y\\x\\0\end{pmatrix}\right\}\] where 
\[\psi(r)=\frac{\sin (\omega r)}{\omega^2r^3}-\frac{\cos (\omega r)}{\omega r^2}.
\]
Finally, for any $i\in\{1,2,3\}$, there holds
{\[ \int_{\B_3}\Vert \U_i\Vert^2=\frac{4\pi}{3\omega^4}\left(\omega^2-\sin^2\omega\right)=\frac{4\pi}{3}\cdot \frac{\sin^2\omega}{\omega^2}=:A^2.\]}
\end{proposition}
The proof of this proposition is essentially contained in \cite{Saks_2013}, which actually gives more general results. For the reader's convenience however, we gather here the main facts necessary to obtain this result.

\paragraph{Eigenvalues of the $\curl$ operator.} A typical approach to Stokes eigenvalue problems is to factorise the vector Laplacian $\Delta$ as $\Delta=\curl(\curl)$ in the case of incompressible flows. 

The spectral problem for the $\curl$ operator is the following: find $(\xi,\u_\xi)\,, \xi \in \C\,, \u_\xi\neq 0$ such that 
\begin{equation}\label{Eq:EigenvalueCurl}
\begin{cases}
\curl(\u_\xi)= \xi \u_\xi&\text{ in }\B_3\,, 
\\ \langle \u_\xi,\nu\rangle=0&\text{ on }\partial \B_3.\end{cases}\end{equation}
In simply connected domains, it is possible to show the following spectral decomposition theorem \cite{zbMATH06248919}:
\begin{theorem}[Spectral decomposition for the $\curl$]\label{Th:SpectralTheorem}
There exist two sequences $\{\mu_{k,\pm}\}_{k\in \N}$ of eigenfunctions ordered as follows
\[ -\infty\longleftarrow \mu_{k+1,-}\leq \mu_{k,-}\leq \dots\leq\mu_{1,-}<0<\mu_{1,+}\leq\dots\leq \mu_{k,+}\leq \mu_{k+1,+}\longrightarrow \infty,\] and the associated eigenfunctions $\{\u_{k,\pm}\}_{k\in \N}$ form a Hilbert basis of the space 
\[ X:=\left\{\u \in L^2(\O)^3\,, \nabla\cdot \u=0\text{ in }\O\,, \langle \u,\nu\rangle=0\text{ on }\partial \O\right\} .\]
\end{theorem}

In this theorem, we use (implicitly) the fact that there is a natural notion of trace for zero-divergence vector fields.

\paragraph{Some explicit computations in the case of the ball.}
We now give a procedure to construct eigenfunctions. All the results stated are contained, in one way or another, in \cite{Saks_2013,Cantarella_2000}. We use the standard spherical coordinates $(r,\theta,\phi)$, and the associated orthonormal basis $(\er,\et,\ep)$. Recall that in this coordinate frame we have 
\[ \begin{cases}x=r\sin\theta \cos(\phi)\,,\\ y=r\sin\theta \sin(\phi)\,,\\ z=r\cos(\phi)\end{cases} \text{ and }\begin{cases}\er= \frac{x \vec{e}_{x} + y \vec{e}_{y} + z \vec{e}_{z}}{\sqrt{x^2 + y^2 + z^2}} \\
  \vec{e}_{\theta} = \frac{\left(x \vec{e}_{x} + y \vec{e}_{y}\right) z - \left(x^2 + y^2\right) \vec{e}_{z}}{\sqrt{x^2 + y^2 + z^2} \sqrt{x^2 + y^2}} \\
 \vec{e}_{\varphi} = \frac{-y \vec{e}_{x} + x \vec{e}_{y}}{\sqrt{x^2 + y^2}}
\end{cases}\]

\begin{lemma}\label{Le:Saks_2013}
Let $(\xi,\u_\xi)$ be a solution of \eqref{Eq:EigenvalueCurl} in $\B_3$ and write, in spherical coordinates,
\[\u_\xi=u_{\xi,r}\vec{e}_r+u_{\xi,\theta}\vec{e}_\theta+u_{\xi,\phi}\vec{e}_\phi.\] 
Then:
\begin{enumerate}
\item There holds $u_{\xi,r}\neq 0$.
\item $ru_{\xi,r}$ is an eigenfunction of the (scalar) Dirichlet-Laplace operator, associated with the eigenvalue $\xi^2$.
\end{enumerate}
\end{lemma}

\begin{proof}[Proof of Lemma \ref{Le:Saks_2013}]
\begin{enumerate}
\item Recall that, in polar coordinates, we have (for a vector field $\u=u_r\er+u_\theta\et+u_\phi\e\phi$),
\begin{multline*}
\curl(\u)=\frac1{r\sin\theta }\left(\frac{\partial (u_\phi\sin\theta )}{\partial\theta}-\frac{\partial u_\theta}{\partial \phi}\right)\er\\+
\frac1r\left(\frac1{\sin\theta }\frac{\partial u_r}{\partial\phi}-\frac{\partial( ru_\phi)}{\partial r}\right)\et
\\+\frac1r\left(\frac{\partial (r u_\theta)}{\partial r}-\frac{\partial u_r}{\partial \theta}\right)\ep.\quad \quad\quad \quad\quad
\end{multline*}
Thus, if $u_{\xi,r}=0$ we deduce that $(u_{\xi,\theta},u_{\xi,\phi})$ solves (in particular) the system
\[
\begin{cases}
-\frac{\partial(ru_\phi)}{\partial r}=\xi ru_\theta\,, 
\\ \frac{\partial (ru_\theta)}{\partial r}=\xi ru_\phi.
\end{cases}\]
Letting $w_{\theta/\phi}:=ru_{\xi,\theta/\phi}$ we obtain 
\[ -\frac{\partial^2 w_{\theta/\phi}}{\partial r^2}=\xi^2 w_{\theta/\phi}\,, w_{\theta/\phi}(0)=w_{\theta/\phi}'(0)=0.\]
The Cauchy-Lipschitz theorem implies $w_{\theta/\phi}\equiv 0$, a contradiction.
\item Let $(\xi,\u_\xi)$ be a solution of \eqref{Eq:EigenvalueCurl}. Define $v:=\langle x,\u\rangle=ru_{\xi,r}$. We have 
\[ -\Delta v=\langle x,\Delta\u\rangle-2\div(\u)=\xi^2\langle x,\u\rangle. \] Here, we used that $\div(\u)=0$, whence $-\Delta \u=\curl(\curl(\u))$. Furthermore, $v=\langle x,\u\rangle=0$ on $\partial \B_3$. As $u_{\xi,r}\neq 0$, the conclusion follows.
\end{enumerate}
\end{proof}

\paragraph{Explicit expression for the lowest eigenvalue of the $\curl$ operator in $\B_3$.}
Let $\omega>0$ be defined as the first positive root of $x=\tan(x)$. We have the following proposition:
\begin{proposition}\label{Pr:LowestCurl}
$\omega$ is the lowest positive eigenvalue of the $\curl$ operator and $-\omega$ is the largest negative eigenvalue of the $\curl$ operator. Both have multiplicity 3.
\end{proposition}

In order to prove this result, recall the following description of the second eigenspace of the Laplacian:
\begin{lemma}\label{Le:Ball}
Let $\psi$ be defined as 
\[\psi(r)=\frac{\sin (\omega r)}{\omega^2r^3}-\frac{\cos (\omega r)}{\omega r^2}.
\] The second eigenvalue $\mu_2$ of the Dirichlet-Laplacian in $\B_3$ is $\omega^2$. Furthermore, $\mu_2(\B_3)$ has multiplicity 3. An orthogonal basis of the eigenspace $E_{\omega^2}^{-\Delta}$ is $(v_1\,, v_2\,, v_3)$, where 
\[ \begin{cases}
v_1:(r,\theta,\phi)\mapsto r\psi(r)\sin\theta \cos\phi\,,
\\ v_2:(r,\theta,\phi)\mapsto r\psi(r)\sin\theta \sin\phi\,,
\\ v_3:(r,\theta,\phi)\mapsto r\psi(r)\cos\theta.
\end{cases}\]

\end{lemma}
To derive Proposition \ref{Pr:LowestCurl}, let $\xi_1$ be the lowest positive eigenvalue of $\curl$, and $\u_1$ be an associated eigenfunction. Then, by Lemma \ref{Le:Saks_2013}, we know that 
\[ \xi_1\geq \sqrt{\mu_1(\B_3)},\] where $\mu_1(\B_3)$ is the first eigenvalue of the Dirichlet Laplacian.
 In the case of the ball, in fact, we have 
 \begin{equation}\label{Eq:EqMagique}
  \xi_1> \sqrt{\mu_1(\B_3)}.\end{equation}
  
  \begin{proof}[Proof of \eqref{Eq:EqMagique}]
  If we had $\xi_1=\sqrt{\mu_1}(\B_3)$, then $\phi:=r{u_{\xi_1,r}}$ is an eigenfunction of the Dirichlet-Laplacian with eigenvalue $\mu_1(\B_3)$. However, $\phi(0)=0$, while, as $\mu_1(\B_3)$ is simple, any eigenfunction associated with $\mu_1$ has a (strict) constant sign in $\B_3$, a contradiction.
  \end{proof}

Proposition \ref{Pr:LowestCurl} is thus implied by the following lemma:
\begin{lemma}\label{Le:Decomposition}
For any $i\in\{1,2,3\}$, there exists a unique $\u_{i,\pm}$ such that:
\begin{enumerate}
\item $(\u_{i,\pm},\pm\omega)$ is a solution of \eqref{Eq:EigenvalueCurl},
\item $\langle x,\u_{i,\pm}\rangle=v_i.$
\end{enumerate}
The $\u_{i,\pm}$ have the following explicit expressions (in spherical coordinates):
\begin{equation}\begin{cases}\u_{1,\pm}=\sin\theta \cos(\phi)r\psi(r)\er+(\frac{F}{r}\cos(\theta)\cos(\phi)\pm \frac{G}{r}\sin(\phi))\et+(\pm \frac{G}{r}\cos(\theta)\cos(\phi)-\frac{F}{r}\sin(\phi))\ep\,, 
\\ \u_{2,\pm}=\sin\theta \sin(\phi)r\psi(r)\er+(\frac{F}{r}\cos(\theta)\sin(\phi)\mp \frac{G}{r}\cos(\phi))\et+(\pm \frac{G}{r}\cos(\theta)\sin(\phi)+\frac{F}{r}\cos(\phi))\ep,
\\\u_{3,\pm}=\cos(\theta)r\psi(r)\er-\frac{F}{r}\sin\theta \et\mp \frac{G}{r}\sin\theta \ep,
\end{cases}
\end{equation}
where $F=F(r)\,, G=G(r)$ satisfy
\begin{equation}
\begin{cases}
F''+\omega F=\frac{d}{dr}\left(\frac1{r^2}\psi(r)\right)\,,
\\ G''+\omega G=-\omega\frac{\psi}{r^2}\,, 
\end{cases}
\end{equation}
The expressions for $F\,, G$ are explicit
\[\begin{cases}
F(r)=\left(\frac{\cos(\omega r)}{\omega r}-\frac{\sin(\omega r)}{\omega^2 r^2}+\sin(\omega r)\right)
\\G(r)=-\left(\frac{\sin(\omega r)}{\omega r}-\cos(\omega r)\right)\,, 
\end{cases}\] and $F(1)\neq 0\,, G(1)=0$.
\end{lemma}
%
%
%
\begin{proof}[Proof of Lemma \ref{Le:Decomposition}]
To study the existence of such functions, we can follow the approach of \cite[Lemma 2]{Saks_2013}. Dropping the indices $i$ in $\u_{i,\pm}$ and $v_i$, the function $\u_\pm=\pm\frac{v}r \er+w \et+z\ep$ must satisfy
\begin{equation}\begin{cases}
\frac1{r\sin\theta }\left(\frac{\partial (z\sin\theta )}{\partial\theta}-\frac{\partial w}{\partial \phi}\right)&=\sqrt{\mu}\frac{ v}r\\
\frac1r\left(\frac1{\sin\theta }\frac{\partial \frac{ v}r}{\partial\phi}-\frac{\partial( rz)}{\partial r}\right)&=\pm\sqrt{\mu}w
\\\frac1r\left(\frac{\partial (r w)}{\partial r}-\frac{\partial \frac{ v}r}{\partial \theta}\right)&=\pm\sqrt{\mu}z.\end{cases}
\end{equation}
We thus obtain 
\[\frac{\partial^2 (rw)}{\partial r^2}+\mu (rw)=\frac{\partial}{\partial r}\left(\frac1{r}\frac{\partial v}{\partial \theta}\right)\pm\frac{\sqrt{\mu}}{r\sin\theta }\frac{\partial v}{\partial \phi}.\] We derive a similar equation for $z$, namely:\[ \frac{\partial^2(rz)}{\partial r^2}+\mu(rz)=\mp\frac{\sqrt{\mu}}r\frac{\partial v}{\partial \theta}+\frac1{\sin\theta }\frac{\partial}{\partial r}\left(\frac1{r}\frac{\partial v}{\partial \phi}\right).\]

Writing these two equations as a system in $\Phi_{w,\pm}:=rw\,, \Phi_{z,\pm}:=rz$, which satisfy $\Phi_{w/z,\pm}(0,\theta,\phi)=\partial_r\Phi_{w/z,\pm}(0,\theta),\phi)=0$, we are thus tasked with solving 
\[\frac{\partial^2 \Phi_{w,\pm}}{\partial r^2}+\omega \Phi_{w,\pm}=\frac{\partial}{\partial r}\left(\frac1r\frac{\partial v}{\partial \theta}\right)\pm\frac{\sqrt{\mu}}{r\sin\theta }\frac{\partial v}{\partial \phi}\] and 
\[\frac{\partial^2 \Phi_{z,\pm}}{\partial r^2}+\omega \Phi_{z,\pm}=\mp\frac{\sqrt{\mu}}r\frac{\partial v}{\partial \theta}+\frac1{\sin\theta }\frac{\partial}{\partial r}\left(\frac1{r}\frac{\partial v}{\partial \phi}\right).\]
However, using the solutions $(F,G)$ given in the statement of the theorem, we can immediately check that the expressions given are the correct ones; indeed, the uniqueness of the pair $(\u_{i,\pm})$ follows from the first point of Lemma \ref{Le:Saks_2013}.
 \end{proof}

An important corollary  of the explicit expression of the first eigenfuntions of the curl is the following:
\begin{corollary}\label{Co:Libre}
Let $(\alpha_{i,\pm})_{1=1,...,3}\in \R^6$ and consider the field
\[ \u=\sum_{i,\pm}\alpha_{i,\pm}\u_{i,\pm}\] where the $u_{i,\pm}$ are given by Lemma \ref{Le:Decomposition}. Then
\[\u=0\text{ on }\partial \B_3\] if, and only if,
\[ \forall i\in \{1,2,3\},\quad  \alpha_{i,+}=-\alpha_{i,-}.\]
\end{corollary}

\paragraph{Proof of Proposition~\ref{Pr:StokesBall}.}
We are now in a position to prove Proposition \ref{Pr:StokesBall}. Consider $\mathbf{U}$ a first eigenfunction of the Dirichlet-Stokes operator, associated with the eigenvalue $\lambda_1(\B_3)$. We already observed that $\lambda_1(\B_3)>\mu_1(\B_3)$. As $\U$ satisfies homogeneous Dirichlet boundary conditions, it follows in particular that 
\[ 
\langle \mathbf{U},\nu\rangle=0\quad \text{ on }\partial \B_3
\] 
and thus $\mathbf{U}\in X$, where $X$ is defined in Theorem \ref{Th:SpectralTheorem}. We can thus decompose $\mathbf{U}$ as 
\[ \mathbf{U}=\sum_{k\in \N} \alpha_{k,\pm}\u_{k,\pm}\] in $X$. Consequently, we deduce that 
\[ \lambda_1(\B_3)\in \{\xi_{k,\pm}^2\,, \xi_{k,\pm}\text{ eigenvalue of the $\curl$ operator}\}.\] In particular, from Lemma \ref{Le:Decomposition}, we have 
\[ \lambda_1(\B_3)\geq \omega^2.\]
However, observe that according to Lemma \ref{Le:Decomposition}, $\u_{1,+}-\u_{1,-}$ is an eigenfunction of the Stokes operator with eigenvalue $\omega^2$. Thus, we deduce that 
\[ \lambda_1(\B_3)=\omega^2\] and, furthermore, that for any associated eigenfunction $\mathbf{U}$, we can decompose $\mathbf{U}$ as 
\[\mathbf{U}=\sum_{i=1}^3\sum_\pm \alpha_{i,\pm} \u_{i,\pm}.\]
Since $\mathbf{U}$ satisfies homogeneous Dirichlet boundary conditions, Corollary \ref{Co:Libre} implies that 
$\mathbf{U}$ writes 
\[\mathbf{U}=\sum_{i=1}^3 \alpha_i\left(\u_{i,+}-\u_{i,-}\right).\] Thus, the eigenspace has dimension at most 3.

We now compute, for each $i\in \{1,2,3\}$, the function $\V_i:=\u_{i,+}-\u_{i,-}$. In spherical coordinates, we have 

\begin{equation}\begin{cases}\V_{1}=2 \frac{G}{r}\sin(\phi)\et+2 \frac{G}{r}\cos(\theta)\cos(\phi)\ep\,, 
\\ \V_{2}=-2 \frac{G}{r}\cos(\phi)\et+2 \frac{G}{r}\cos(\theta)\sin(\phi)\ep,
\\\V_{3}=-2 \frac{G}{r}\sin\theta \ep.
\end{cases}\end{equation}
Now, to get a nicer expression in cartesian coordinates $(\vec{e}_x,\vec{e}_y,\vec{e}_z)$, recall that 
\[\er=\frac1r\left(x\vec{e}_x+y\vec{e}_y+z\vec{e}_z\right),\quad \et=\frac{xz\vec{e}_x+yz\vec{e}_y-(x^2+y^2)\vec{e}_z}{r\sqrt{x^2+y^2}},\quad  \ep=\frac{-y \vec{e}_x+x\vec{e}_y}{\sqrt{x^2+y^2}}.
\]
 Furthermore, observe that
 \[ 
 \frac{G(r)}{r^2}=\psi(r).\] We thus obtain, in cartesian coordinates:
  
 \[\V_1=2\psi(r)\begin{pmatrix}0\\z\\-y\end{pmatrix}\,,\V_2=2\psi(r)\begin{pmatrix}-z\\0\\x\end{pmatrix}\,, \V_3=2\psi(r)\begin{pmatrix}-y\\x\\0\end{pmatrix}.
 \]
 In particular, defining $\U_i:=\frac{\V_i}2$, we observe that $\U_1\,, \U_2\,, \U_3$ are orthogonal, whence the conclusion: the eigenspace is three dimensional, and we have an orthogonal basis of it.

It remains to compute the normalization constant $A$, using a spherical change of coordinates and explicit computations. Let us provide some details hereafter. One has
\begin{eqnarray*} 
\Vert \mathbf{U}_i\Vert^2_{L^2(\B_3)} &=& \int_{\B_3}(x^2+y^2)\psi(r)^2\, dxdydz=\int_0^{2\pi}\int_0^\pi\int_0^1 r^4\sin^3\theta \psi(r)^2\, drd\theta d\varphi\\
&=& 2\pi \int_0^\pi \sin^3\theta\, d\theta\int_0^1 \left(\frac{\sin (\omega r)}{\omega^2r}-\frac{\cos (\omega r)}{\omega}\right)^2\, dr\\
&=& \frac{8\pi}{3}\left(\frac{2\cos^2\omega+\omega^2+\omega\cos\omega\sin\omega-2}{2\omega^4}\right).
\end{eqnarray*}
The desired expression follows easily by using that {$\omega \cos\omega=\sin \omega$}.

\section{Existence of an optimal shape under a box constraint}\label{Ap:DalMaso}

In this appendix, we prove the following result: 
\begin{proposition}\label{Pr:Box}
Let  $D\subset \R^d$ be a fixed compact domain of $\R^d$ and define the admissible class \begin{eqnarray*}
\widehat{\mathscr{O}}_{V_0}  = &\left\{\Omega \textrm{ quasi-open included in }D \,, |\Omega|\leq V_0 \right\}.
\end{eqnarray*}
The shape optimisation problem
\begin{equation}\inf_{\O\in \widehat{\mathscr{O}}_{V_0}}\lambda(\O)\end{equation} has a solution.\end{proposition}
\begin{proof}[Proof of Proposition \ref{Pr:Box}]
We introduce, for $\Omega\in \widehat{\mathcal{O}}_{V_0}$, the quantity $\widehat{\lambda_1}(V_0\color{black})$ given by
\begin{equation}\label{hatLamb1V}
\widehat{\lambda_1}(V_0)=\inf_{\substack{\mathbf{u}\in W^{1,2}_0(D;\R^3)\\ \int_D \mathbf{u}\cdot \nabla v=0\ \forall v\in H^1(D)\\ |\Omega_{\mathbf{u}}|\leq V_0\\\int_{\Omega_{\mathbf{u}}} \Vert \u\Vert^2=1}}{\int_D \Vert \nabla \mathbf{u}\Vert^2}\end{equation}
where $\Omega_{\mathbf{u}}$ denotes the quasi open set $\{\mathbf{u}\neq 0\}$. 

Let $\{\u_k\}_{k\in\mathbb{N}}$ be minimising for Problem \eqref{hatLamb1V}. For any $k\in \N$, let $\Omega_k:=\{\u_k\neq 0\}$. In view of the Poincar\'e inequality \cite[Lemma~4.5.3]{zbMATH06838450},  $\{\u_k\}_{k\in\N}$ is bounded in $W^{1,2}_0(D;\R^3)$. It follows that there exists $\mathbf{u}^\star\in W^{1,2}_0(D;\R^3)$ such that $\{\u_k\}_{k\in\N}$ converges to $\mathbf{u}^\star$ up to a (not relabelled) subsequence, weakly in $W^{1,2}_0$ and strongly in $L^2$. Thus, $\{u_k\}_{k\in \N}$ converges almost everywhere to $\mathbf{u}^\star$, and therefore
\[|\Omega_{\mathbf{u^\star}}|\leq \liminf_{k\to +\infty}|\Omega_k|\leq V_0
\]
where $\Omega_{\mathbf{u^\star}}$ denotes the quasi-open set $\{\mathbf{u^\star}\neq 0\}$. We infer that $\Omega_{\mathbf{u^\star}}$ belongs to $\widehat{\mathscr{O}}_{V_0}$. Furthermore, since $\mathbf{u^\star}=0$ quasi-everywhere on $D\backslash \Omega_{\mathbf{u}^\star}$, by weak $H^1$-convergence, one has
$$
\int_{\Omega_{\mathbf{u}^\star}}\Vert \nabla \mathbf{u}^\star\Vert ^2 \leq \liminf_{k\to +\infty}\int_{\Omega_k}\Vert \nabla \u_k\Vert^2 = \widehat{\lambda_1}(V_0) 
\leq \int_{\Omega_{\mathbf{u}^\star}} \Vert \nabla \mathbf{u}^\star\Vert ^2
$$
by minimality, whence the equality of these quantities. By strong convergence in $L^2(D;\R^3)$ and weak convergence in $W^{1,2}_0(D;\R^3)$, there holds 
\[\int_D \Vert\mathbf{u}^*\Vert^2=1\] and, for any $v\in W^{1,2}(D)$, 
\[
0=\lim_{k\to +\infty}\int_D\langle \u_k, \nabla v\rangle=\int_D\langle \mathbf{u}^\star, \nabla v\rangle.\] Thus, $\O_{\u^*}$ is a solution of \eqref{hatLamb1V}. The proof is now concluded.
\end{proof}
\bibliographystyle{abbrv}

\bibliography{BibHmp}

\end{document}